\documentclass{article}
\usepackage[text={170mm,254mm}]{geometry}
\usepackage{graphicx}
\usepackage[T1]{fontenc}
\usepackage[utf8]{inputenc}

\usepackage{amssymb, amsthm, amsfonts, amsmath, enumitem, mathtools}
\usepackage{booktabs, array, diagbox, xcolor, multicol, rotating, stmaryrd}

\usepackage{algorithmicx}
\usepackage{algpseudocode}

\usepackage[hyphens]{url}
\usepackage[colorlinks=true,citecolor=black,linkcolor=black,urlcolor=blue]{hyperref}
\usepackage{cite}

\usepackage{subcaption}
\usepackage{float}

\newcommand{\arxiv}[2]{\href{https://arxiv.org/abs/#1}{\texttt{arXiv:#1}} \texttt{[#2]}}
\newcommand{\doi}[1]{\url{https://doi.org/#1}}

\theoremstyle{plain}
\newtheorem{theorem}{Theorem}[section]
\newtheorem{lemma}[theorem]{Lemma}
\newtheorem{conjecture}[theorem]{Conjecture}
\newtheorem{proposition}[theorem]{Proposition}
\newtheorem{corollary}[theorem]{Corollary}
\theoremstyle{definition}
\newtheorem{definition}[theorem]{Definition}
\newtheorem{example}[theorem]{Example}

\usepackage{listings}
\definecolor{mauve}{rgb}{0.58,0,0.82}
\definecolor{dkgreen}{rgb}{0,0.6,0}
\lstdefinestyle{pitonche} {
    language = Python,
    basicstyle = footnotesizettfamily,
    showspaces = false,
    showstringspaces = false,
    breakautoindent = true,
    flexiblecolumns = true,
    keepspaces = true,
    stepnumber = 1,
    xleftmargin = 0pt
}
\usepackage{tikz}

\newcommand{\Iverson}[1]{\left\llbracket #1 \right\rrbracket}

\DeclareMathOperator{\Sym}{Sym}

\let\le=\leqslant
\let\ge=\geqslant

\usepackage{authblk}

\title{Some results on small ordered and cyclic Ramsey numbers}

\author[1,2,3]{Nino Bašić}
\author[1,4]{Ivan Damnjanović\thanks{Corresponding author. Email:\ \texttt{ivan.damnjanovic@elfak.ni.ac.rs} (I.\ Damnjanović).}}
\author[5]{Dragan Stevanović\thanks{On leave from the Mathematical Institute of the Serbian Academy of Sciences and Arts.}}
\author[6]{Ivan Stošić}

\affil[1]{FAMNIT, University of Primorska, Koper, Slovenia, Glagoljaška 8, Koper, 6000, Slovenia}
\affil[2]{IAM, University of Primorska, Koper, Slovenia, Muzejski trg 2, Koper, 6000, Slovenia}
\affil[3]{Institute of Mathematics, Physics and Mechanics, Ljubljana, Slovenia, Jadranska ulica 19, Ljubljana, 1000, Slovenia}
\affil[4]{Faculty of Electronic Engineering, University of Niš, Aleksandra Medvedeva 4, Niš, 18104, Serbia}
\affil[5]{College of Integrative Studies, Abdullah Al Salem University, Firdous Street, Block 3, Khaldiya, 72303, Kuwait}
\affil[6]{Faculty of Sciences and Mathematics, University of Niš, Višegradska 33, Niš, 18106, Serbia}

\date{}

\begin{document}

\maketitle

\begin{abstract}
Let $k \in \mathbb{N}$ and let $H_1, H_2, \ldots, H_k$ be simple graphs such that for each $j \in \{ 1, 2, \ldots, k \}$, the vertex set of $H_j$ is $\{ 0, 1, 2, \ldots, n_j - 1 \}$ for some $n_j \in \mathbb{N}$. The ordered Ramsey number $R_\mathrm{ord}(H_1, H_2, \ldots, H_k)$ is the smallest $n \in \mathbb{N}$ for which every $k$-edge-coloring of the complete graph on the vertex set $\{ 0, 1, 2, \ldots, n - 1 \}$ contains $H_j$ as a monochromatic subgraph of color $j$ for some $j \in \{ 1, 2, \ldots, k \}$, with the vertices appearing in the same order as in $H_j$. Inspired by the work of Poljak, we apply the Kissat SAT solver to determine new small two-color ordered Ramsey numbers of various classes of graphs: monotone paths, monotone cycles, alternating paths, stars, complete graphs and nested matchings. In addition, we introduce the cyclic Ramsey numbers $R_\mathrm{cyc}(H_1, H_2, \ldots, H_k)$ as a natural relaxation of the ordered Ramsey numbers, and once again use Kissat to determine various such numbers for the two-color case. By observing structural patterns in the computational results, we determine all ordered or cyclic Ramsey numbers for several pairs of classes of graphs. Furthermore, we obtain some bounds on ordered and cyclic Ramsey numbers where one argument is a connected graph, while the other is a monotone path or a monotone cycle. We also explore how reinforcement learning can be used through the recently developed Reinforcement Learning for Graph Theory (RLGT) framework to obtain lower bounds on ordered and cyclic Ramsey numbers. Finally, we introduce the permutational Ramsey numbers to show how the different Ramsey-type formulations involving standard, ordered and cyclic Ramsey numbers can be unified within a group-theoretic framework.
\end{abstract}



\section{Introduction}

In this paper, we consider all graphs to be undirected, simple and finite, and we assume each graph to have a vertex set of the form $\{ 0, 1, 2, \ldots, n - 1 \}$ for some $n \in \mathbb{N}$. Moreover, we use $|G|$ to denote the order, i.e., the number of vertices, of a graph $G$, and $K_n$ to denote the complete graph of order $n$. We consider two graphs $G_1$ and $G_2$ to be equal, and denote this by $G_1 = G_2$, if they have the same vertex set and the same edge set. In addition, we say that two graphs $G_1$ and $G_2$ are \emph{rotationally isomorphic}, and write $G_1 \cong_\rho G_2$, if there exists an isomorphism $\varphi \colon V(G_1) \to V(G_2)$ of the form $\varphi(v) = (v + s) \bmod |G_1|$ for some $s \in \mathbb{Z}$.

By Ramsey's theorem \cite{Ramsey1930} from 1930, for any choice of graphs $H_1, H_2, \ldots, H_k$ with $k \in \mathbb{N}$, every $k$-edge-coloring of $K_n$ with sufficiently large $n \in \mathbb{N}$ contains $H_j$ as a monochromatic subgraph of color $j$ for some $j \in \{ 1, 2, \ldots, k \}$. The smallest $n$ for which the asserted claim holds is then referred to as the \emph{Ramsey number} $R(H_1, H_2, \ldots, H_k)$. Despite its simple formulation, the problem of computing Ramsey numbers turns out to be quite challenging. For instance, the Ramsey number $R(K_5, K_5)$ is still not known as of today \cite{Radziszowski}. Over the last century, many mathematicians have taken an interest in studying Ramsey numbers and other Ramsey-type results pertaining to generalizations of graphs, geometric configurations, and sequences; see \cite{ConFoxSu2015, ErSze1935, FraPaReiRo2018, GraRothSpe1990, LiLin2022, Morris2026, NeRo1990, Nguyen2014, Radziszowski} and references therein.

One of the branches of Ramsey theory that has recently attracted considerable interest from researchers is the study of ordered Ramsey numbers. For any graphs $H_1, H_2, \ldots, H_k$ with $k \in \mathbb{N}$, the \emph{ordered Ramsey number} $R_\mathrm{ord}(H_1, H_2, \ldots, H_k)$ is defined as the smallest $n \in \mathbb{N}$ for which every $k$-edge-coloring of $K_n$ contains $H_j$ as a monochromatic subgraph of color $j$ for some $j \in \{ 1, 2, \ldots, k \}$, with the vertices appearing in the same order as in $H_j$. The definition is essentially the same as that of the Ramsey number $R(H_1, H_2, \ldots, H_k)$, with the difference that the embedding $\varphi \colon V(H_j) \to V(K_n)$ is additionally required to be increasing. We note that Ramsey's theorem implies that the ordered Ramsey numbers are well-defined. In fact, the following result is immediate.

\begin{proposition}
For any graphs $H_1, H_2, \ldots, H_k$ with $k \in \mathbb{N}$, we have
\[
    R(H_1, H_2, \ldots, H_k) \le R_\mathrm{ord}(H_1, H_2, \ldots, H_k) \le R(K_{|H_1|}, K_{|H_2|}, \ldots, K_{|H_k|}) .
\]
In particular, $R_\mathrm{ord}(K_{n_1}, K_{n_2}, \ldots, K_{n_k}) = R(K_{n_1}, K_{n_2}, \ldots, K_{n_k})$ for any $k \in \mathbb{N}$ and $n_1, n_2, \ldots, n_k \in \mathbb{N}$.
\end{proposition}

The systematic study of ordered Ramsey numbers was initiated independently by Balko, Cibulka, Král and Kynčl \cite{BalCiKralKyn2015, BalCiKralKyn2020}, and Conlon, Fox, Lee and Sudakov \cite{ConFoxLeeSu2017}. Over the next decade, many researchers have taken an interest in studying the asymptotic properties of various classes of ordered Ramsey numbers; see, e.g., \cite{BalCiKralKyn2015, BalCiKralKyn2020, BalJeVa2019, BalPolj2023, BalPolj2024, BraMoSuWi2024, ConFoxLeeSu2017, CoxSto2016, GiJaJa2024, GiJinSu2024, Grinerova2024, MuSuk2024, NeiWest2019, Poljak2023, Rohatgi2018}, and the survey \cite{Balko2025} and references therein. Here, we take a different direction and are primarily interested in computing the exact values of small ordered Ramsey numbers.

For each $n \in \mathbb{N}$, let $P_n^\mathrm{mon}$ denote the \emph{monotone path} graph of order $n$, i.e., the graph of order $n$ in which two vertices are adjacent if and only if they are consecutive integers. Using a similar strategy as when proving the well-known Erdős--Szekeres theorem on monotonic subsequences \cite{ErSze1935}, it is not difficult to obtain the following result; see \cite{BalCiKralKyn2015, BalCiKralKyn2020, ChoPo2002, ConFoxLeeSu2017, MiStoWest2015}.

\begin{theorem}[\hspace{1sp}{\cite[Proposition~10]{BalCiKralKyn2020}}]\label{pmon_pmon_ord_th}
    For any $k \in \mathbb{N}$ and $n_1, n_2, \ldots, n_k \in \mathbb{N}$, we have
    \[
        R_\mathrm{ord}(P_{n_1}^\mathrm{mon}, P_{n_2}^\mathrm{mon}, \ldots, P_{n_k}^\mathrm{mon}) = 1 + \prod_{j = 1}^k (n_j - 1) .
    \]
\end{theorem}

\noindent
Moreover, it turns out that in the two-color case, if we replace one of the monotone paths with a complete graph, the ordered Ramsey number stays the same, as implicitly shown by Károlyi, Pach and Tóth; see \cite{BalCiKralKyn2020, KaPaTo1997}.

\begin{theorem}[\hspace{1sp}{\cite[Lemma~18]{BalCiKralKyn2020}}]\label{pmon_k_th}
    For any $a, b \in \mathbb{N}$, we have $R_\mathrm{ord}(P_a^\mathrm{mon}, K_b) = 1 + (a - 1)(b - 1)$.
\end{theorem}

Now, for any $n \ge 3$, let $C_n^\mathrm{mon}$ denote the \emph{monotone cycle} graph of order $n$, which arises from $P_n^\mathrm{mon}$ by adding the edge $\{ 0, n - 1 \}$. Also, for convenience, let $C_2^\mathrm{mon} \coloneqq K_2$. The following theorem was recently proved by Balko et al.\ \cite{BalCiKralKyn2020}, thereby extending an earlier result by Károlyi, Pach, Tóth and Valtr \cite{KaPaToVal1998}.

\begin{theorem}[\hspace{1sp}{\cite[Theorem~4]{BalCiKralKyn2020}}]\label{cyc_cyc_ord_th}
    For any $a, b \ge 2$, we have $R_\mathrm{ord}(C_a^\mathrm{mon}, C_b^\mathrm{mon}) = 2ab - 3a - 3b + 6$.
\end{theorem}

Another recent set of results concerning the computation of small ordered Ramsey numbers involves alternating paths. We define the \emph{alternating path} graph of order $n \in \mathbb{N}$, denoted by $P_n^\mathrm{alt}$, as the path graph of order $n$ with the underlying path $(0, n - 1, 1, n - 2, 2, n - 3, \ldots, \lfloor \frac{n}{2} \rfloor)$; see Figure \ref{alt_fig}. Balko et al.\ \cite{BalCiKralKyn2020} initially determined $R_\mathrm{ord}(P_n^\mathrm{alt}, P_n^\mathrm{alt})$ for $n \le 8$, and obtained lower bounds for $n \in \{ 9, 10, \ldots, 13 \}$. Afterwards, Poljak \cite{Poljak2023} determined $R_\mathrm{ord}(P_n^\mathrm{alt}, P_n^\mathrm{alt})$ for $n \in \{ 9, 10 \}$ and improved the lower bound for $n = 13$. Kucheriya, Lo, Petr, Sgueglia and Yan \cite{KuLoPeSguYan2026} later found a general upper bound on $R_\mathrm{ord}(P_n^\mathrm{alt}, P_n^\mathrm{alt})$ that led to the exact computation of these ordered Ramsey numbers for $n \in \{ 11, 12, 13 \}$. All of these results can be succinctly summarized in Table \ref{alt_tab}.

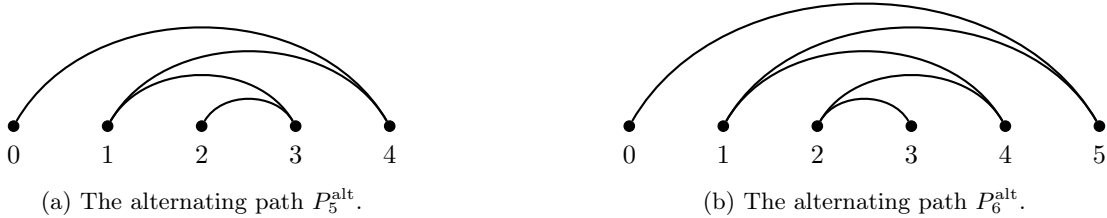
\begin{figure}[t]
\centering
\subcaptionbox{The alternating path $P_5^\mathrm{alt}$.}[0.47\textwidth]
{
    \centering
    \begin{tikzpicture}
        \tikzstyle{vertex}=[draw,circle,font=\scriptsize,minimum size=4pt,inner sep=1pt,fill=black]
        \tikzstyle{edge}=[draw,thick]

        \foreach \i in {0,1,2,3,4} {
            \node[vertex] (v\i) at ({1.25*\i}, 0) {};
            \node[below=4pt] at (v\i) {$\i$};
        }

        \path[edge, bend left=60] (v0) to (v4);
        \path[edge, bend right=60] (v4) to (v1);
        \path[edge, bend left=60] (v1) to (v3);
        \path[edge, bend right=60] (v3) to (v2);
    \end{tikzpicture}
}
\qquad
\subcaptionbox{The alternating path $P_6^\mathrm{alt}$.}[0.47\textwidth]
{
    \centering
    \begin{tikzpicture}
        \tikzstyle{vertex}=[draw,circle,font=\scriptsize,minimum size=4pt,inner sep=1pt,fill=black]
        \tikzstyle{edge}=[draw,thick]

        \foreach \i in {0,1,2,3,4,5} {
            \node[vertex] (v\i) at ({1.25*\i}, 0) {};
            \node[below=4pt] at (v\i) {$\i$};
        }

        \path[edge, bend left=60] (v0) to (v5);
        \path[edge, bend right=60] (v5) to (v1);
        \path[edge, bend left=60] (v1) to (v4);
        \path[edge, bend right=60] (v4) to (v2);
        \path[edge, bend left=60] (v2) to (v3);
    \end{tikzpicture}
}
\caption{The alternating path graphs of orders five and six.}
\label{alt_fig}
\end{figure}

\begin{table}
\centering
\renewcommand{\arraystretch}{1.2}
\begin{tabular}{|c||rrrrrrrrrrrrr|}
\hline 
$n$ & $1$ & $2$ & $3$ & $4$ & $5$ & $6$ & $7$ & $8$ & $9$ & $10$ & $11$ & $12$ & $13$\\
\hline 
\hline 
$R_\mathrm{ord}(P_n^\mathrm{alt}, P_n^\mathrm{alt})$ & $1$ & $2$ & $4$ & $7$ & $9$ & $12$ & $15$ & $17$ & $20$ & $23$ & $25$ & $28$ & $31$\\
\hline
\end{tabular}
\caption{The ordered Ramsey numbers $R_\mathrm{ord}(P_n^\mathrm{alt}, P_n^\mathrm{alt})$ for $n \le 13$. Sources:\ \cite{BalCiKralKyn2020, KuLoPeSguYan2026, Poljak2023}.}
\label{alt_tab}
\end{table}

Here, our goal is to systematically determine new small two-color ordered Ramsey numbers involving specific classes of graphs, such as monotone paths, monotone cycles and alternating paths, as well as other classes of graphs appearing in literature, such as stars, complete graphs and nested matchings. Besides diagonal numbers, we are also interested in computing off-diagonal numbers corresponding to either two different graphs of the same class, or two graphs from different classes. Inspired by the work of Poljak \cite{Poljak2020}, we apply the Kissat SAT solver \cite{Kissat, KissatRepo} to compute all the desired numbers. Our contribution extends the known results on the computation of ordered Ramsey numbers of graphs with a specific structure, and complements the existing results concerning the computation of ordered Ramsey numbers of graphs up to a fixed order \cite{BroLiMiPu2025, OveAlmCoLa2024}.

Apart from studying ordered Ramsey numbers, we introduce the cyclic Ramsey numbers, which arise naturally when replacing the total ordering of vertices with a cyclic ordering. Using Kissat, we compute cyclic Ramsey numbers alongside the ordered ones for the previously considered classes of graphs.

For any graphs $H_1, H_2, \ldots, H_k$ with $k \in \mathbb{N}$, we define the \emph{cyclic Ramsey number} $R_\mathrm{cyc}(H_1, H_2, \ldots, H_k)$ as the smallest $n \in \mathbb{N}$ for which every $k$-edge-coloring of $K_n$ contains $H_j$ as a monochromatic subgraph of color $j$ for some $j \in \{1, 2, \ldots, k \}$, with the vertices appearing in the same cyclic order as in $H_j$. More precisely, the embedding $\varphi \colon V(H_j) \to V(K_n)$ is additionally required to be increasing up to a cyclic permutation, i.e., there must exist $t \in V(H_j)$ such that $(\varphi(t), \varphi(t + 1), \ldots, \varphi(|H_j| - 1), \varphi(0), \varphi(1), \ldots, \varphi(t - 1))$ is an increasing sequence.

In the same way that ordered Ramsey numbers correspond to vertex sets with a total ordering, cyclic Ramsey numbers correspond to vertex sets with a cyclic ordering. For this reason, it is natural to consider this relaxation of the ordered Ramsey numbers. We observe that Ramsey's theorem implies that the cyclic Ramsey numbers are well-defined for any $k \in \mathbb{N}$ and any choice of graphs $H_1, H_2, \ldots, H_k$. The next result immediately follows.

\begin{proposition}\label{cyc_basic_prop}
For any graphs $H_1, H_2, \ldots, H_k$ with $k \in \mathbb{N}$, we have
\[
    R(H_1, H_2, \ldots, H_k) \le R_\mathrm{cyc}(H_1, H_2, \ldots, H_k) \le R_\mathrm{ord}(H_1, H_2, \ldots, H_k) \le R(K_{|H_1|}, K_{|H_2|}, \ldots, K_{|H_k|}) .
\]
In particular, $R_\mathrm{ord}(K_{n_1}, K_{n_2}, \ldots, K_{n_k}) = R_\mathrm{cyc}(K_{n_1}, K_{n_2}, \ldots, K_{n_k}) = R(K_{n_1}, K_{n_2}, \ldots, K_{n_k})$ for any $k \in \mathbb{N}$ and $n_1, n_2, \ldots, n_k \in \mathbb{N}$.
\end{proposition}

Finally, we demonstrate how reinforcement learning (RL) can be applied through the recently developed Reinforcement Learning for Graph Theory (RLGT) framework \cite{RLGT, RLGTRepo} to obtain lower bounds on ordered and cyclic Ramsey numbers. Although the SAT solver based approach outperforms RL, the partial success of the RLGT framework indicates that RL could be a useful asset in tackling similar combinatorial problems in the future.

In Section \ref{sc_prel}, we define all the remaining classes of graphs whose ordered and cyclic Ramsey numbers we study. In addition, we give a few initial observations on the cyclic Ramsey numbers. Afterwards, in Section~\ref{sc_sat}, we describe how the ordered and cyclic Ramsey number problems can be reduced to Boolean satisfiability (SAT) problems and how the Kissat SAT solver was configured to obtain the desired results. The computational results are then presented in Section \ref{sc_results}, alongside the theorems and conjectures derived from these results. These include the determination of all ordered or cyclic Ramsey numbers for several pairs of classes of graphs, as well as some bounds on the ordered and cyclic Ramsey numbers where one argument is a connected graph, while the other is a monotone path or a monotone cycle. Subsequently, Section~\ref{sc_rl} discusses how RL can be applied to obtain lower bounds on ordered and cyclic Ramsey numbers and how the RLGT framework was configured to achieve this goal. Finally, Section \ref{sc_conclusion} ends the paper with a brief conclusion and discusses possible directions for future research. In particular, we introduce the permutational Ramsey numbers to show how the different Ramsey-type formulations involving standard, ordered and cyclic Ramsey numbers can be unified within a group-theoretic framework. The source code used to obtain the computational results can be found in \cite{GitHub}, together with all the constructed graphs that yield the desired lower bounds.

\section{Preliminaries}\label{sc_prel}

In the present section, we first define a few additional classes of graphs that we will be interested in when computing ordered and cyclic Ramsey numbers. To begin, we consider a \emph{reverse alternating path} graph of order $n \in \mathbb{N}$, denoted by $P_n^\mathrm{ralt}$, to be the path graph of order $n$ with the underlying path $(n - 1, 0, n - 2, 1, n - 3, \linebreak 2, \ldots, \lceil \frac{n}{2} \rceil - 1)$; see Figure \ref{ralt_fig}. It is natural to consider these graphs since they arise by reflecting the alternating paths, as noted by Kucheriya et al.\ \cite{KuLoPeSguYan2026}. Although $R_\mathrm{ord}(P_a^\mathrm{ralt}, P_b^\mathrm{ralt}) = R_\mathrm{ord}(P_a^\mathrm{alt}, P_b^\mathrm{alt})$ and $R_\mathrm{cyc}(P_a^\mathrm{ralt}, P_b^\mathrm{ralt}) = R_\mathrm{cyc}(P_a^\mathrm{alt}, P_b^\mathrm{alt})$ obviously hold for any $a, b \in \mathbb{N}$, it does make sense to investigate reverse alternating paths when combining them with graphs from other classes. In particular, we will be interested in computing the off-diagonal numbers $R_\mathrm{ord}(P_a^\mathrm{alt}, P_b^\mathrm{ralt})$ and $R_\mathrm{cyc}(P_a^\mathrm{alt}, P_b^\mathrm{ralt})$, which could behave differently from those involving only alternating paths.

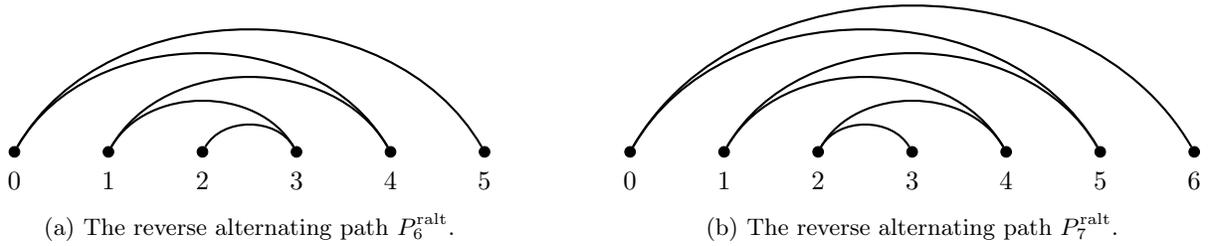
\begin{figure}[t]
\centering
\subcaptionbox{The reverse alternating path $P_6^\mathrm{ralt}$.}[0.44\textwidth]
{
    \centering
    \begin{tikzpicture}
        \tikzstyle{vertex}=[draw,circle,font=\scriptsize,minimum size=4pt,inner sep=1pt,fill=black]
        \tikzstyle{edge}=[draw,thick]

        \foreach \i in {0,1,2,3,4,5} {
            \node[vertex] (v\i) at ({1.25*\i}, 0) {};
            \node[below=4pt] at (v\i) {$\i$};
        }

        \path[edge, bend right=60] (v5) to (v0);
        \path[edge, bend left=60] (v0) to (v4);
        \path[edge, bend right=60] (v4) to (v1);
        \path[edge, bend left=60] (v1) to (v3);
        \path[edge, bend right=60] (v3) to (v2);
    \end{tikzpicture}
}
\qquad
\subcaptionbox{The reverse alternating path $P_7^\mathrm{ralt}$.}[0.50\textwidth]
{
    \centering
    \begin{tikzpicture}
        \tikzstyle{vertex}=[draw,circle,font=\scriptsize,minimum size=4pt,inner sep=1pt,fill=black]
        \tikzstyle{edge}=[draw,thick]

        \foreach \i in {0,1,2,3,4,5,6} {
            \node[vertex] (v\i) at ({1.25*\i}, 0) {};
            \node[below=4pt] at (v\i) {$\i$};
        }

        \path[edge, bend right=60] (v6) to (v0);
        \path[edge, bend left=60] (v0) to (v5);
        \path[edge, bend right=60] (v5) to (v1);
        \path[edge, bend left=60] (v1) to (v4);
        \path[edge, bend right=60] (v4) to (v2);
        \path[edge, bend left=60] (v2) to (v3);
    \end{tikzpicture}
}
\caption{The reverse alternating path graphs of orders six and seven.}
\label{ralt_fig}
\end{figure}

We define the \emph{start-central star} graph of order $n \in \mathbb{N}$, denoted by $S_n^\mathrm{sc}$, as the star graph of order $n$ in which vertex $0$ is adjacent to all the other vertices. As noted by Balko et al.\ \cite{BalCiKralKyn2020}, it is straightforward to compute the ordered Ramsey numbers where all the arguments are start-central stars.

\begin{proposition}[\hspace{1sp}{\cite[Observation~12]{BalCiKralKyn2020}}]\label{prel_star_2}
    For any $k \in \mathbb{N}$ and $n_1, n_2, \ldots, n_k \ge 2$, we have
    \[
        R_\mathrm{ord}(S_{n_1}^\mathrm{sc}, S_{n_2}^\mathrm{sc}, \ldots, S_{n_k}^\mathrm{sc}) = 2 + \sum_{j = 1}^k (n_j - 2).
    \]
\end{proposition}

\noindent
Although the ordered Ramsey numbers of stars with different centers have been studied \cite{ChoPo2002}, in this work we restrict our attention to start-central stars. 

We observe that any two stars of the same order are equivalent for the cyclic Ramsey number problem, regardless of their centers. For this reason, when computing cyclic Ramsey numbers, we can just write $S_n$, which denotes any star of order $n$, instead of $S_n^\mathrm{sc}$. The next result directly follows.

\begin{proposition}\label{prel_star}
    For any $k \in \mathbb{N}$ and $n_1, n_2, \ldots, n_k \in \mathbb{N}$, we have
    \[
        R_\mathrm{cyc}(S_{n_1}, S_{n_2}, \ldots, S_{n_k}) = R(S_{n_1}, S_{n_2}, \ldots, S_{n_k}) .
    \]
\end{proposition}

\noindent
Finally, for any even $n \ge 2$, let $M_n^\mathrm{nest}$ denote the \emph{nested matching} graph of order $n$, i.e., the $1$-regular graph of order $n$ in which each vertex $v$ is adjacent only to $n - 1 - v$; see Figure \ref{nm_fig}.

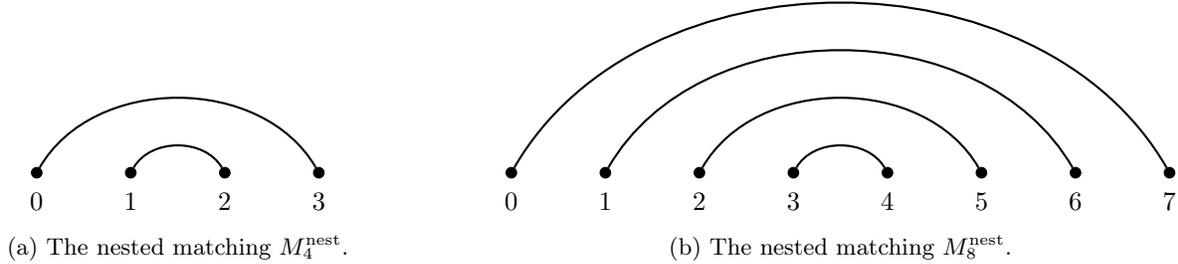
\begin{figure}[t]
\centering
\subcaptionbox{The nested matching $M_4^\mathrm{nest}$.}[0.32\textwidth]
{
    \centering
    \begin{tikzpicture}
        \tikzstyle{vertex}=[draw,circle,font=\scriptsize,minimum size=4pt,inner sep=1pt,fill=black]
        \tikzstyle{edge}=[draw,thick]

        \foreach \i in {0,1,2,3} {
            \node[vertex] (v\i) at ({1.25*\i}, 0) {};
            \node[below=4pt] at (v\i) {$\i$};
        }

        \path[edge, bend left=60] (v0) to (v3);
        \path[edge, bend left=60] (v1) to (v2);
    \end{tikzpicture}
}
\qquad
\subcaptionbox{The nested matching $M_8^\mathrm{nest}$.}[0.62\textwidth]
{
    \centering
    \begin{tikzpicture}
        \tikzstyle{vertex}=[draw,circle,font=\scriptsize,minimum size=4pt,inner sep=1pt,fill=black]
        \tikzstyle{edge}=[draw,thick]

        \foreach \i in {0,1,2,3,4,5,6,7} {
            \node[vertex] (v\i) at ({1.25*\i}, 0) {};
            \node[below=4pt] at (v\i) {$\i$};
        }

        \path[edge, bend left=60] (v0) to (v7);
        \path[edge, bend left=60] (v1) to (v6);
        \path[edge, bend left=60] (v2) to (v5);
        \path[edge, bend left=60] (v3) to (v4);
    \end{tikzpicture}
}
\caption{The nested matching graphs of orders four and eight.}
\label{nm_fig}
\end{figure}

It is easy to see that a cyclic Ramsey number is invariant under rotation of the vertices of any of its argument graphs. This yields the following result.

\begin{proposition}\label{cyc_rot_inv_prop}
    For some $k \in \mathbb{N}$, let $H_1, H_2, \ldots, H_k$ and $H_1', H_2', \ldots, H_k'$ be graphs such that $H_j' \cong_\rho H_j$ for each $j \in \{ 1, 2, \ldots, k \}$. Then
    \[
        R_\mathrm{cyc}(H_1', H_2', \ldots, H_k') = R_\mathrm{cyc}(H_1, H_2, \ldots, H_k) .
    \]
\end{proposition}

\noindent
As an immediate corollary of Propositions \ref{cyc_basic_prop} and \ref{cyc_rot_inv_prop}, we obtain the next result.

\begin{corollary}
    For graphs $H_1, H_2, \ldots, H_k$ with $k \in \mathbb{N}$, let $\mathcal{H}'$ be the set of all the $k$-tuples $(H_1', H_2', \ldots, H_k')$ such that $H_j' \cong_\rho H_j$ for each $j \in \{ 1, 2, \ldots, k \}$. Then
    \begin{equation}\label{rot_ineq}
        R_\mathrm{cyc}(H_1, H_2, \ldots, H_k) \le \min_{(H_1', H_2', \ldots, H_k') \in \mathcal{H}'} R_\mathrm{ord}(H_1', H_2', \ldots, H_k') .
    \end{equation}
\end{corollary}

\noindent
We conclude the section with an example showing that the gap between the left-hand and right-hand side of~\eqref{rot_ineq} can be wide.

\begin{example}
    Let $n \ge 3$ and for each $j \in \{ 0, 1, 2, \ldots, n - 1 \}$, let $P^\mathrm{qmon}_{n, j}$ be the quasi-monotone path that arises from $C_n^\mathrm{mon}$ by deleting the edge $\{ j, (j + 1) \bmod n \}$. Observe that these graphs form a single equivalence class under vertex rotation, and that $P^\mathrm{qmon}_{n, n - 1} = P^\mathrm{mon}_n$. By using Kissat as described in Section \ref{sc_sat}, we can compute $R_\mathrm{ord}(P^\mathrm{qmon}_{5, a}, P^\mathrm{qmon}_{5, b})$ for each $a, b \in \{ 0, 1, 2, 3, 4 \}$; see Table \ref{pqmon_pqmon_tab}.

    Therefore, if we let $k = 2$ and $H_1 = H_2 = P^\mathrm{mon}_5$, then the right-hand side of \eqref{rot_ineq} equals $\min \{ 16, 17 \} = 16$. On the other hand, a quick computation via Kissat also gives $R_\mathrm{cyc}(P^\mathrm{mon}_5, P^\mathrm{mon}_5) = 13$; see Table \ref{pmon_pmon_cyc_tab}. This demonstrates that the left-hand side of \eqref{rot_ineq} can be notably lower than the right-hand side. If we additionally take into consideration the standard Ramsey numbers, a well-known result by Gerencsér and Gyárfás \cite{GeGya1967} gives $R(P_5, P_5) = 6$. This shows that the cyclic Ramsey number problem does not simply come down to computing ordered Ramsey numbers, but instead relies on a structural pattern entirely different from both the standard and the ordered Ramsey number problems. \hfill$\Diamond$
\end{example}
 
\begin{table}[H]
\centering
\footnotesize
\begin{tabular}{|c||rrrrr|}
\hline
\backslashbox{$a$}{$b$} & $0$ & $1$ & $2$ & $3$ & $4$\\
\hline
\hline
$0$ & $16$ & $16$ & $16$ & $16$ & $17$\\
$1$ & & $17$ & $17$ & $16$ & $17$\\
$2$ & & & $17$ & $16$ & $17$\\
$3$ & & & & $16$ & $17$\\
$4$ & & & & & $17$\\
\hline
\end{tabular}
\caption{The ordered Ramsey numbers $R_\mathrm{ord}(P^\mathrm{qmon}_{5, a}, P^\mathrm{qmon}_{5, b})$ for any $0 \le a \le b \le 4$.}
\label{pqmon_pqmon_tab}
\end{table}

\section{Reduction to SAT}\label{sc_sat}

Let $H_1$ and $H_2$ be two given graphs and let $n \in \mathbb{N}$. The problem of finding a $2$-edge-coloring of $K_n$ that does not contain $H_j$ as a monochromatic subgraph of color $j$ for any $j \in \{ 1, 2 \}$ can naturally be modeled as a SAT problem. This approach is especially convenient when the embedding $\varphi \colon V(H_j) \to V(K_n)$ is additionally required to be increasing, i.e., in the context of ordered Ramsey number problems, as explained by Poljak \cite{Poljak2020}. Here, we demonstrate how the same idea can be applied to cyclic Ramsey number problems. For completeness, we also include Poljak's original approach.

Consider a $2$-edge-coloring $G$ of a complete graph $K_n$, and for each $u, v \in V(K_n)$ with $u < v$, let $x_{u, v}$ be a Boolean variable that indicates whether the edge $\{ u, v \}$ is colored with color $2$. Also, for convenience, let $x_{v, u} \coloneqq x_{u, v}$. Then, the nonexistence of $H_1$ as a monochromatic subgraph of $G$ of color $1$, with the vertices appearing in the same order as in $H_1$, is equivalent to
\begin{equation}\label{cl1}
    \bigwedge_{(w_0, w_1, w_2, \ldots, w_{|H_1| - 1}) \in \mathcal{O}_1} \bigvee_{uv \in E(H_1)} x_{w_u, w_v} ,
\end{equation}
where $\mathcal{O}_1$ comprises all the increasing $|H_1|$-tuples of elements in $V(K_n)$. Similarly, the nonexistence of $H_2$ as a monochromatic subgraph of $G$ of color $2$, with the vertices appearing in the same order as in $H_2$, is equivalent to
\begin{equation}\label{cl2}
    \bigwedge_{(w_0, w_1, w_2, \ldots, w_{|H_2| - 1}) \in \mathcal{O}_2} \bigvee_{uv \in E(H_2)} \neg x_{w_u, w_v} ,
\end{equation}
where $\mathcal{O}_2$ comprises all the increasing $|H_2|$-tuples of elements in $V(K_n)$. By combining \eqref{cl1} and \eqref{cl2}, we obtain a SAT problem with $\binom{n}{2}$ variables whose resolution yields a bound on $R_\mathrm{ord}(H_1, H_2)$. If the SAT problem has a solution, then $R_\mathrm{ord}(H_1, H_2) \ge n + 1$, and otherwise, we have $R_\mathrm{ord}(H_1, H_2) \le n$.

The above strategy can be adjusted to cover cyclic Ramsey number problems, so that the forbidden monochromatic subgraphs are those with the vertices appearing in the same cyclic order, rather than in the same order. In this case, the SAT clauses \eqref{cl1} and \eqref{cl2} should be replaced with
\[
    \bigwedge_{(w_0, w_1, w_2, \ldots, w_{|H_1| - 1}) \in \mathcal{C}_1} \bigvee_{uv \in E(H_1)} x_{w_u, w_v} \quad \mbox{and} \quad \bigwedge_{(w_0, w_1, w_2, \ldots, w_{|H_2| - 1}) \in \mathcal{C}_2} \bigvee_{uv \in E(H_2)} \neg x_{w_u, w_v} ,
\]
respectively, where $\mathcal{C}_1$ (resp.\ $\mathcal{C}_2$) comprises all the $|H_1|$-tuples (resp.\ $|H_2|$-tuples) of elements in $V(K_n)$ that are increasing up to a cyclic permutation. If the obtained SAT problem has a solution, then $R_\mathrm{cyc}(H_1, H_2) \ge n + 1$, and otherwise, $R_\mathrm{cyc}(H_1, H_2) \le n$. These SAT formulations can also be extended to more than two edge colors, but we focus on two-color Ramsey problems.

All the computed ordered and cyclic Ramsey numbers presented in Section \ref{sc_results} were obtained by solving the corresponding SAT problems. The variables were arranged in row-major order, i.e., as
\[
    x_{0, 1}, x_{0, 2}, \ldots, x_{0, n - 1}, x_{1, 2}, x_{1, 3}, \ldots, x_{1, n - 1}, x_{2, 3}, x_{2, 4}, \ldots, x_{2, n - 1}, \ldots, x_{n - 2, n - 1}.
\]
We used the Kissat SAT solver, primarily due to its strong performance in the main track of the SAT Competition 2024. The source code used for generating SAT instances and processing the results is available in \cite{GitHub}, together with all the computed data, organized as follows.
\begin{enumerate}[label=\textbf{(\arabic*)}]
    \item The \texttt{src} folder contains \texttt{C++} source code for generating SAT instances and an independent \texttt{Python} script for the same purpose. The folder also includes \texttt{bash} scripts for automatizing the process of ordered and cyclic Ramsey number computation, a \texttt{Python} script for parsing the Kissat output, and \texttt{C++} code for Ramsey score computation. The Ramsey scores can then optionally be used by a \texttt{Python} RL training script; see Section \ref{sc_rl}.
    \item The \texttt{kissat\_output} folder contains all output files produced by running Kissat on the generated SAT instances.
    \item The \texttt{parsed\_graphs} folder contains the graphs extracted from the Kissat output files using the \texttt{Python} parsing script. These graphs are stored in the \texttt{graph6} format \cite{McKayGraph6, McKayPip2014} and provide lower bounds for the computed ordered and cyclic Ramsey numbers.
    \item The \texttt{generated\_tables} folder contains automatically generated \LaTeX{} code used to produce the tables in Section \ref{sc_results}. This code was also generated using the \texttt{Python} parsing script.
\end{enumerate}
Further technical details on compiling and running Kissat, and on the structure and usage of the provided code, can be found in \cite{GitHub}.

\section{Results}\label{sc_results}

In this section, we provide the computed values for small two-color ordered and cyclic Ramsey numbers of various classes of graphs, alongside the theorems and conjectures derived from the computational results. We begin with the bounds on the ordered and cyclic Ramsey numbers where one argument is a connected graph, while the other is a monotone path or a monotone cycle. The first of these results generalizes Theorem \ref{pmon_k_th} and deals with the ordered Ramsey numbers of connected graphs versus monotone paths.

\begin{theorem}\label{pmon_con_th}
    Let $H$ be a connected graph of order $a \in \mathbb{N}$. Then for any $b \in \mathbb{N}$, we have $R_\mathrm{ord}(H, P_b^\mathrm{mon}) = 1 + (a - 1)(b - 1)$.
\end{theorem}
\begin{proof}
    The theorem trivially holds if $\min \{ a, b \} = 1$, so we assume that $a, b \ge 2$. Since $H$ is a subgraph of $K_a$, Theorem \ref{pmon_k_th} implies $R_\mathrm{ord}(H, P_b^\mathrm{mon}) \le 1 + (a - 1)(b - 1)$. Therefore, to complete the proof, it suffices to construct a $2$-edge-coloring of $K_n$ with $n = (a - 1)(b - 1)$ that contains neither $H$ as a monochromatic subgraph of color~$1$ nor $P_b^\mathrm{mon}$ as a monochromatic subgraph of color $2$, with the vertices required to appear in the same order.

    Let $f \colon \{ 0, 1, 2, \ldots, n - 1 \} \to \{ 0, 1, 2, \ldots, b - 2 \}$ be the function defined by $f(x) \coloneqq \lfloor \frac{x}{a - 1} \rfloor$. Now, let $G$ be the $2$-edge-coloring of $K_n$ such that the edge between any two distinct vertices $u, v \in \{ 0, 1, 2, \ldots, n - 1 \}$ is colored with color $1$ if and only if $f(u) = f(v)$. Suppose that $G$ contains $P_b^\mathrm{mon}$ as a monochromatic subgraph of color $2$ with the increasing embedding $\varphi \colon V(P_b^\mathrm{mon}) \to V(G)$. Then, for any $v \in \{ 0, 1, \ldots, b - 2 \}$, we have $f(\varphi(v)) \neq f(\varphi(v + 1))$, hence $f(\varphi(v + 1)) > f(\varphi(v))$. This means that the numbers $f(\varphi(0)), f(\varphi(1)), \ldots, f(\varphi(b - 1))$ are mutually distinct, which is impossible because the range of $f$ contains only $b - 1$ numbers.

    Now, suppose that $G$ contains $H$ as a monochromatic subgraph of color $1$ with the increasing embedding $\varphi \colon V(H) \to V(G)$. In this case, for any vertices $u, v \in V(H)$ adjacent in $H$, we have $f(\varphi(u)) = f(\varphi(v))$. Since $H$ is connected, this implies that $f$ maps all the numbers $\varphi(0), \varphi(1), \ldots, \varphi(a - 1)$ to the same value. However, this is impossible since each element in the codomain of $f$ is the image of exactly $a - 1$ numbers.
\end{proof}

\noindent
As direct corollaries of Theorem \ref{pmon_con_th}, we obtain the following results on ordered Ramsey numbers.

\begin{corollary}
    For any $a, b \in \mathbb{N}$ and path graph $P_a$ of order $a$, we have $R_\mathrm{ord}(P_a, P_b^\mathrm{mon}) = 1 + (a - 1)(b - 1)$.
\end{corollary}
\begin{corollary}\label{cyc_pmon_cor}
    For any $a \ge 3$ and $b \in \mathbb{N}$, and cycle graph $C_a$ of order $a$, we have $R_\mathrm{ord}(C_a, P_b^\mathrm{mon}) = 1 + (a - 1) \linebreak (b - 1)$.
\end{corollary}
\begin{corollary}\label{star_pmon_cor}
    For any $a, b \in \mathbb{N}$ and star graph $S_a$ of order $a$, we have $R_\mathrm{ord}(S_a, P_b^\mathrm{mon}) = 1 + (a - 1)(b - 1)$.
\end{corollary}

With these results in mind, we consider monotone paths only in the study of cyclic Ramsey numbers, apart from the case of combining monotone paths with nested matchings, since $M_n^\mathrm{nest}$ is disconnected for any even $n \ge 4$. Next, we provide a theorem that mirrors Theorem \ref{pmon_con_th} and gives a lower bound on the cyclic Ramsey numbers of connected graphs versus monotone paths.

\begin{theorem}\label{pmon_con_th_2}
    Let $H$ be a connected graph of order $a \in \mathbb{N}$. Then for any $b \in \mathbb{N}$, we have
    \[
        1 + (a - 1)(b - 2) \le R_\mathrm{cyc}(H, P_b^\mathrm{mon}) \le 1 + (a - 1)(b - 1).
    \]
\end{theorem}
\begin{proof}
    The upper bound follows directly from Proposition \ref{cyc_basic_prop} and Theorem \ref{pmon_con_th}, while the lower bound is trivial when $a = 1$ or $b \le 2$. Now, suppose that $a \ge 2$ and $b \ge 3$. To complete the proof, it remains to construct a $2$-edge-coloring of $K_n$ with $n = (a - 1)(b - 2)$ that contains neither $H$ as a monochromatic subgraph of color~$1$ nor $P_b^\mathrm{mon}$ as a monochromatic subgraph of color $2$, with the vertices required to appear in the same cyclic order.

    Let $f \colon \{ 0, 1, 2, \ldots, n - 1 \} \to \{ 0, 1, 2, \ldots, b - 3 \}$ be the function defined by $f(x) \coloneqq \lfloor \frac{x}{a - 1} \rfloor$. Now, let $G$ be the $2$-edge-coloring of $K_n$ such that the edge between any two distinct vertices $u, v \in \{ 0, 1, 2, \ldots, n - 1 \}$ is colored with color~$1$ if and only if $f(u) = f(v)$. As in Theorem~\ref{pmon_con_th}, one can show that $G$ does not contain $H$ as a monochromatic subgraph of color~$1$, so we omit the proof.

    Now, by way of contradiction, assume that $G$ contains $P_b^\mathrm{mon}$ as a monochromatic subgraph of color $2$ with the embedding $\varphi \colon V(P_b^\mathrm{mon}) \to V(G)$ that is increasing up to a cyclic permutation. Since there exists $t \in \{ 0, 1, 2, \ldots, b - 1 \}$ such that $(\varphi(t), \varphi(t + 1), \ldots, \varphi(b - 1), \varphi(0), \varphi(1), \ldots, \varphi(t - 1))$ is an increasing sequence, it follows that $(f(\varphi(t)), f(\varphi(t + 1)), \ldots, f(\varphi(b - 1)), f(\varphi(0)), f(\varphi(1)), \ldots, f(\varphi(t - 1)))$ is nondecreasing. Using the same argument as in Theorem \ref{pmon_con_th}, we conclude that
    \[
        f(\varphi(t)) < f(\varphi(t + 1)) < \cdots < f(\varphi(b - 1)) \le f(\varphi(0)) < f(\varphi(1)) < \cdots < f(\varphi(t - 1)),
    \]
    hence these $b$ numbers have at least $b - 1$ different values. This yields a contradiction since the range of $f$ contains only $b - 2$ elements.
\end{proof}

We mention in passing that the upper bound in Theorem \ref{pmon_con_th_2} is sharp; see Corollary \ref{k_pmon_cyc_cor}. In addition, the lower bound also appears to be sharp; see Conjecture \ref{s_pmon_cyc_conj}. The following theorem can also be proved using the same technique as in the proofs of Theorems \ref{pmon_con_th} and \ref{pmon_con_th_2}, so we omit the proof.

\begin{theorem}\label{one_more_th}
    Let $H$ be a connected graph of order $a \in \mathbb{N}$. Then for any $b \ge 2$, we have $R_\mathrm{cyc}(H, C_b^\mathrm{mon}) \ge 1 + (a - 1)(b - 1)$.
\end{theorem}

We note that the bound in Theorem \ref{one_more_th} is sharp; see Theorem \ref{cmon_pmon_cyc_th}. In the remainder of the section, we present the computational results obtained via Kissat. While doing so, we assume that both argument graphs have at least three vertices, since otherwise it is trivial to obtain the desired ordered or cyclic Ramsey number.

\subsection{Paths and cycles}

We first investigate the ordered and cyclic Ramsey numbers involving paths and cycles. The computed cyclic Ramsey numbers of monotone paths versus monotone paths are shown in Table \ref{pmon_pmon_cyc_tab}. These results, together with Theorem \ref{pmon_con_th_2}, lead to the following conjecture.

\begin{conjecture}\label{pmon_pmon_cyc_conj}
    For any $b \ge a \ge 3$, we have $R_\mathrm{cyc}(P_a^\mathrm{mon}, P_b^\mathrm{mon}) = 1 + (a - 1)(b - 2)$.
\end{conjecture}

\begin{table}[H]
\centering
\footnotesize
\begin{tabular}{|c||rrrrrrrrrrrrr|}
\hline
\backslashbox{$a$}{$b$} & $3$ & $4$ & $5$ & $6$ & $7$ & $8$ & $9$ & $10$ & $11$ & $12$ & $13$ & $14$ & $15$\\
\hline
\hline
$3$ & $3$ & $5$ & $7$ & $9$ & $11$ & $13$ & $15$ & $17$ & $19$ & $21$ & $23$& $\ge 25$& $\ge 26$\\
$4$ & & $7$ & $10$ & $13$ & $16$ & $19$ & $22$ & $25$& $\ge 25$& $\ge 26$ & & &\\
$5$ & & & $13$ & $17$ & $21$ & $25$& $\ge 28$ & & & & & &\\
$6$ & & & & $21$ & $26$& $\ge 31$ & & & & & & &\\
$7$ & & & && $\ge 31$ & $\ge 34$ & & & & & & &\\
$8$ & & & & && $\ge 34$ & & & & & & &\\
\hline
\end{tabular}
\caption{Cyclic Ramsey numbers $R_\mathrm{cyc}(P_a^\mathrm{mon}, P_b^\mathrm{mon})$ with $3 \le a \le b$, $a \le 8$ and $b \le 15$.}
\label{pmon_pmon_cyc_tab}
\end{table}

Motivated by the Ramsey numbers summarized in Table \ref{alt_tab}, we study the ordered and cyclic Ramsey numbers of alternating paths versus alternating paths. The computed values are reported in Tables \ref{palt_palt_ord_tab} and \ref{palt_palt_cyc_tab}. Although it does not seem easy to notice a pattern for the ordered Ramsey numbers, the cyclic numbers exhibit a behavior leading to the following conjecture.

\begin{conjecture}
    For any $a, b \ge 2$, we have $R_\mathrm{cyc}(P_a^\mathrm{alt}, P_b^\mathrm{alt}) = a + b - 2 - (ab \bmod 2)$.
\end{conjecture}

\noindent
As in the case of ordered Ramsey numbers, the cyclic Ramsey numbers of monotone paths seem to be larger and grow more rapidly than those of alternating paths.

\begin{table}[H]
\centering
\footnotesize
\resizebox{\textwidth}{!}{
\begin{tabular}{|c||rrrrrrrrrrrrrrrrrrrr|}
\hline
\backslashbox{$a$}{$b$} & $3$ & $4$ & $5$ & $6$ & $7$ & $8$ & $9$ & $10$ & $11$ & $12$ & $13$ & $14$ & $15$ & $16$ & $17$ & $18$ & $19$ & $20$ & $21$ & $22$\\
\hline
\hline
$3$ & $4$ & $6$ & $7$ & $8$ & $9$ & $11$ & $12$ & $13$ & $14$ & $15$ & $16$ & $18$ & $19$ & $20$ & $21$ & $22$ & $23$ & $24$ & $25$ & $27$\\
$4$ & & $7$ & $8$ & $10$ & $11$ & $12$ & $14$ & $15$ & $16$ & $17$ & $19$ & $20$ & $21$& $\ge 22$& $\ge 23$& $\ge 24$& $\ge 26$ & & &\\
$5$ & & & $9$ & $11$ & $12$ & $14$ & $15$ & $17$ & $18$ & $19$ & $20$& $\ge 22$& $\ge 23$& $\ge 24$& $\ge 25$ & & & & &\\
$6$ & & & & $12$ & $14$ & $15$ & $17$ & $18$ & $19$& $\ge 21$& $\ge 22$& $\ge 23$& $\ge 24$ & & & & & & &\\
$7$ & & & & & $15$ & $16$ & $18$& $\ge 19$& $\ge 21$& $\ge 22$ & & & & & & & & & &\\
$8$ & & & & & & $17$& $\ge 19$& $\ge 20$& $\ge 22$ & & & & & & & & & & &\\
\hline
\end{tabular}
}
\caption{Ordered Ramsey numbers $R_\mathrm{ord}(P_a^\mathrm{alt}, P_b^\mathrm{alt})$ with $3 \le a \le b$, $a \le 8$ and $b \le 22$.}
\label{palt_palt_ord_tab}
\end{table}

\begin{table}[H]
\centering
\footnotesize
\resizebox{\textwidth}{!}{
\begin{tabular}{|c||rrrrrrrrrrrrrrrrrr|}
\hline
\backslashbox{$a$}{$b$} & $3$ & $4$ & $5$ & $6$ & $7$ & $8$ & $9$ & $10$ & $11$ & $12$ & $13$ & $14$ & $15$ & $16$ & $17$ & $18$ & $19$ & $20$\\
\hline
\hline
$3$ & $3$ & $5$ & $5$ & $7$ & $7$ & $9$ & $9$ & $11$ & $11$ & $13$ & $13$ & $15$ & $15$ & $17$ & $17$ & $19$ & $19$ & $21$\\
$4$ & & $6$ & $7$ & $8$ & $9$ & $10$ & $11$ & $12$ & $13$ & $14$ & $15$ & $16$& $\ge 17$& $\ge 18$& $\ge 19$& $\ge 20$& $\ge 21$& $\ge 22$\\
$5$ & & & $7$ & $9$ & $9$ & $11$ & $11$ & $13$& $\ge 13$& $\ge 15$& $\ge 15$& $\ge 17$& $\ge 17$& $\ge 19$& $\ge 19$& $\ge 21$& $\ge 21$& $\ge 23$\\
$6$ & & & & $10$ & $11$ & $12$ & $13$& $\ge 14$& $\ge 15$& $\ge 16$& $\ge 17$& $\ge 18$& $\ge 19$& $\ge 20$& $\ge 21$& $\ge 22$& $\ge 23$& $\ge 24$\\
$7$ & & & && $11$& $\ge 13$& $\ge 13$& $\ge 15$& $\ge 15$& $\ge 17$ & & & & & & & &\\
$8$ & & & & && $\ge 14$& $\ge 15$& $\ge 16$ & & & & & & & & & &\\
\hline
\end{tabular}
}
\caption{Cyclic Ramsey numbers $R_\mathrm{cyc}(P_a^\mathrm{alt}, P_b^\mathrm{alt})$ with $3 \le a \le b$, $a \le 8$ and $b \le 20$.}
\label{palt_palt_cyc_tab}
\end{table}

We now turn to mixed comparisons and first consider the cyclic Ramsey numbers of alternating paths versus monotone paths; see Table \ref{palt_pmon_cyc_tab}. In this case, the pattern seems to differ depending on whether the monotone path argument is $P_3^\mathrm{mon}$. Since $P_3^\mathrm{mon} \cong_\rho P_3^\mathrm{alt}$, we focus only on the case where the monotone path is of order at least four. In this situation, the following conjecture naturally arises.

\begin{conjecture}
    For any $a \ge 3$ and $b \ge 4$, we have $R_\mathrm{cyc}(P_a^\mathrm{alt}, P_b^\mathrm{mon}) = 1 + (a - 1)(b - 2)$.
\end{conjecture}

\begin{table}[H]
\centering
\footnotesize
\begin{tabular}{|c||rrrrrrrrrrrrr|}
\hline
\backslashbox{$a$}{$b$} & $3$ & $4$ & $5$ & $6$ & $7$ & $8$ & $9$ & $10$ & $11$ & $12$ & $13$ & $14$ & $15$\\
\hline
\hline
$3$ & $3$ & $5$ & $7$ & $9$ & $11$ & $13$ & $15$ & $17$ & $19$ & $21$ & $23$& $\ge 25$& $\ge 26$\\
$4$ & $5$ & $7$ & $10$ & $13$ & $16$ & $19$ & $22$& $\ge 25$& $\ge 27$ & & & &\\
$5$ & $5$ & $9$ & $13$ & $17$ & $21$ & $25$ & $\ge 29$ & & & & & &\\
$6$ & $7$ & $11$ & $16$ & $21$ & $26$& $\ge 30$& & & & & & &\\
$7$ & $7$ & $13$ & $19$ & $25$& $\ge 30$& $\ge 32$& & & & & & &\\
$8$ & $9$ & $15$ & $22$& $\ge 29$& $\ge 31$& $\ge 32$ & & & & & & &\\
$9$ & $9$ & $17$& $\ge 25$& $\ge 29$& & & & & & & & &\\
$10$ & $11$ & $19$& $\ge 28$ & & & & & & & & & &\\
$11$ & $11$& $\ge 21$& & & & & & & & & & & \\
$12$ & $13$& $\ge 23$& & & & & & & & & & &\\
$13$ & $13$& $\ge 25$& & & & & & & & & & &\\
$14$ & $15$& $\ge 26$ & & & & & & & & & & &\\
$15$ & $15$& & & & & & & & & & & &\\
$16$ & $17$& & & & & & & & & & & &\\
\hline
\end{tabular}
\caption{Cyclic Ramsey numbers $R_\mathrm{cyc}(P_a^\mathrm{alt}, P_b^\mathrm{mon})$ with $3 \le a \le 16$ and $3 \le b \le 15$.}
\label{palt_pmon_cyc_tab}
\end{table}

As the next step, we extend the analysis to reverse alternating paths and investigate the ordered and cyclic Ramsey numbers of alternating paths versus reverse alternating paths. The resulting values are given in Tables~\ref{palt_pralt_ord_tab} and \ref{palt_pralt_cyc_tab}. As in the case of alternating paths versus alternating paths, the ordered Ramsey numbers seem to follow a pattern that is difficult to decipher, while the cyclic numbers appear to be the same as when both arguments are alternating paths.

\begin{conjecture}\label{palt_pralt_cyc_conj}
    For any $a, b \ge 2$, we have $R_\mathrm{cyc}(P_a^\mathrm{alt}, P_b^\mathrm{ralt}) = a + b - 2 - (ab \bmod 2)$.
\end{conjecture}

\begin{table}[H]
\centering
\footnotesize
\begin{tabular}{|c||rrrrrrrrrrrrrrrrrr|}
\hline
\backslashbox{$a$}{$b$} & $3$ & $4$ & $5$ & $6$ & $7$ & $8$ & $9$ & $10$ & $11$ & $12$ & $13$ & $14$ & $15$ & $16$ & $17$ & $18$ & $19$ & $20$\\
\hline
\hline
$3$ & $5$ & $6$ & $7$ & $8$ & $10$ & $11$ & $12$ & $13$ & $14$ & $15$ & $17$ & $18$ & $19$ & $20$ & $21$ & $22$ & $23$ & $24$\\
$4$ & & $7$ & $8$ & $10$ & $11$ & $12$ & $14$ & $15$ & $16$ & $17$ & $19$ & $20$ & $21$& $\ge 22$ & & & &\\
$5$ & & & $10$ & $11$ & $13$ & $14$ & $15$ & $17$ & $18$& $\ge 19$& $\ge 21$ & & & & & & &\\
$6$ & & & & $12$ & $14$ & $15$ & $17$& $\ge 18$& $\ge 19$ & & & & & & & & &\\
$7$ & & & & & $15$ & $16$& $\ge 18$& $\ge 19$ & & & & & & & & & &\\
$8$ & & & & & & $18$& $\ge 19$ & & & & & & & & & & &\\
\hline
\end{tabular}
\caption{Ordered Ramsey numbers $R_\mathrm{ord}(P_a^\mathrm{alt}, P_b^\mathrm{ralt})$ with $3 \le a \le b$, $a \le 8$ and $b \le 20$.}
\label{palt_pralt_ord_tab}
\end{table}

\begin{table}[H]
\centering
\footnotesize
\begin{tabular}{|c||rrrrrrrrrrrrrrrrrr|}
\hline
\backslashbox{$a$}{$b$} & $3$ & $4$ & $5$ & $6$ & $7$ & $8$ & $9$ & $10$ & $11$ & $12$ & $13$ & $14$ & $15$ & $16$ & $17$ & $18$ & $19$ & $20$\\
\hline
\hline
$3$ & $3$ & $5$ & $5$ & $7$ & $7$ & $9$ & $9$ & $11$ & $11$ & $13$ & $13$ & $15$ & $15$ & $17$ & $17$ & $19$ & $19$ & $21$\\
$4$ & & $6$ & $7$ & $8$ & $9$ & $10$ & $11$ & $12$ & $13$ & $14$ & $15$ & $16$ & $\ge 17$ & $\ge 18$ & $\ge 19$ & & &\\
$5$ & & & $7$ & $9$ & $9$ & $11$ & $11$ & $13$ & $\ge 13$ & $\ge 15$ & $\ge 15$ & & & & & & &\\
$6$ & & & & $10$ & $11$ & $12$ & $13$ & $\ge 14$ & $\ge 15$ & $\ge 16$ & & & & & & & &\\
$7$ & & & & & $\ge 11$ & $\ge 13$ & $\ge 13$ & $\ge 15$ & & & & & & & & & &\\
$8$ & & & & & & $\ge 14$ & $\ge 15$ & $\ge 16$ & & & & & & & & & &\\
\hline
\end{tabular}
\caption{Cyclic Ramsey numbers $R_\mathrm{cyc}(P_a^\mathrm{alt}, P_b^\mathrm{ralt})$ with $3 \le a \le b$, $a \le 8$ and $b \le 20$.}
\label{palt_pralt_cyc_tab}
\end{table}

Theorem \ref{cyc_cyc_ord_th} determines all the ordered Ramsey numbers of monotone cycles. Since any monotone cycle is rotationally isomorphic only to itself, the ordered and cyclic Ramsey numbers coincide when all arguments are monotone cycles. This leads to the next corollary.

\begin{corollary}
    For any $a, b \ge 2$, we have $R_\mathrm{cyc}(C_a^\mathrm{mon}, C_b^\mathrm{mon}) = 2ab - 3a - 3b + 6$.
\end{corollary}

We conclude this subsection by combining monotone cycles with monotone and alternating paths. For the cyclic Ramsey numbers of monotone cycles versus monotone paths, we directly determine all of these numbers instead of presenting a table with computed values.

\begin{theorem}\label{cmon_pmon_cyc_th}
    For any $a \ge 2$ and $b \in \mathbb{N}$, we have $R_\mathrm{cyc}(C_a^\mathrm{mon}, P_b^\mathrm{mon}) = 1 + (a - 1)(b - 1)$.
\end{theorem}
\begin{proof}
    Theorem \ref{one_more_th} directly gives $R_\mathrm{cyc}(C_a^\mathrm{mon}, P_b^\mathrm{mon}) \ge 1 + (a - 1)(b - 1)$. By Proposition~\ref{cyc_basic_prop} and Corollary~\ref{cyc_pmon_cor}, we have $R_\mathrm{cyc}(C_a^\mathrm{mon}, P_b^\mathrm{mon}) \le R_\mathrm{ord}(C_a^\mathrm{mon}, P_b^\mathrm{mon}) = 1 + (a - 1)(b - 1)$. Therefore, $R_\mathrm{cyc}(C_a^\mathrm{mon}, P_b^\mathrm{mon}) = 1 + (a - 1) \linebreak (b - 1)$.
\end{proof}

\noindent
As is turns out, the ordered and cyclic Ramsey numbers of monotone cycles versus monotone paths coincide.

Finally, Tables \ref{cmon_palt_ord_tab} and \ref{cmon_palt_cyc_tab} contain the computed ordered and cyclic Ramsey numbers of monotone cycles versus alternating paths, respectively. Although there is not much computational data available, it makes sense to pose the following two conjectures.

\begin{conjecture}\label{cmon_palt_ord_conj}
    For any $a \ge 3$ and $b \ge 2$, we have $R_\mathrm{ord}(C_a^\mathrm{mon}, P_b^\mathrm{alt}) = \lceil (a - \frac{1}{2})(b - 1) \rceil$.
\end{conjecture}

\begin{conjecture}\label{cmon_palt_cyc_conj}
    For any $a \ge 2$ and $b \in \mathbb{N}$, we have $R_\mathrm{cyc}(C_a^\mathrm{mon}, P_b^\mathrm{alt}) = 1 + (a - 1)(b - 1)$.
\end{conjecture}

\noindent
Note that the cyclic Ramsey numbers of monotone cycles versus alternating paths appear to be the same as those of monotone cycles versus monotone paths.

\begin{table}[H]
\centering
\footnotesize
\begin{tabular}{|c||rrrrrrrr|}
\hline
\backslashbox{$a$}{$b$} & $3$ & $4$ & $5$ & $6$ & $7$ & $8$ & $9$ & $10$\\
\hline
\hline
$3$ & $5$ & $8$ & $10$ & $13$ & $15$ & $18$ & $\ge 20$ & $\ge 23$\\
$4$ & $7$ & $11$ & $14$ & $18$ & $\ge 21$ & $\ge 25$ & $\ge 28$ &\\
$5$ & $9$ & $14$ & $18$ & $23$ & $\ge 27$ & $\ge 32$ & $\ge 34$ &\\
$6$ & $11$ & $17$ & $22$ & $\ge 28$ & $\ge 33$ & $\ge 38$ & &\\
$7$ & $13$ & $20$ & $\ge 26$ & $\ge 33$ & $\ge 37$ & & &\\
$8$ & $15$ & $23$ & $\ge 30$ & $\ge 37$ & $\ge 42$ & & &\\
$9$ & $17$ & $\ge 26$ & $\ge 33$ & $\ge 37$ & & & &\\
$10$ & $19$ & $\ge 29$ & $\ge 34$ & & & & &\\
\hline
\end{tabular}
\caption{Ordered Ramsey numbers $R_\mathrm{ord}(C_a^\mathrm{mon}, P_b^\mathrm{alt})$ with $3 \le a, b \le 10$.}
\label{cmon_palt_ord_tab}
\end{table}

\begin{table}[H]
\centering
\footnotesize
\begin{tabular}{|c||rrrrrrrr|}
\hline
\backslashbox{$a$}{$b$} & $3$ & $4$ & $5$ & $6$ & $7$ & $8$ & $9$ & $10$\\
\hline
\hline
$3$ & $5$ & $7$ & $9$ & $11$ & $13$ & $15$ & $\ge 17$ & $\ge 19$\\
$4$ & $7$ & $10$ & $13$ & $16$ & $19$ & $\ge 22$ & $\ge 25$ & $\ge 25$\\
$5$ & $9$ & $13$ & $17$ & $21$ & $\ge 25$ & $\ge 29$ & $\ge 27$ &\\
$6$ & $11$ & $16$ & $21$ & $\ge 26$ & $\ge 30$ & $\ge 32$ & &\\
$7$ & $13$ & $19$ & $\ge 25$ & $\ge 31$ & $\ge 31$ & & &\\
$8$ & $15$ & $22$ & $\ge 26$ & $\ge 30$ & $\ge 29$ & & &\\
$9$ & $17$ & $\ge 25$ & $\ge 27$ & $\ge 30$ & & & &\\
$10$ & $19$ & $\ge 27$ & $\ge 28$ & & & & &\\
\hline
\end{tabular}
\caption{Cyclic Ramsey numbers $R_\mathrm{cyc}(C_a^\mathrm{mon}, P_b^\mathrm{alt})$ with $3 \le a, b \le 10$.}
\label{cmon_palt_cyc_tab}
\end{table}

\subsection{Stars versus paths and cycles}

A well-known result by Burr and Roberts \cite{BurrRo1973} determines all standard Ramsey numbers whose arguments are star graphs.

\begin{theorem}[\hspace{1sp}{\cite{BurrRo1973}}]\label{star_cool_th}
    For any $k \in \mathbb{N}$ and $n_1, n_2, \ldots, n_k \ge 2$, let $t$ be the number of odd integers among $n_1, n_2, \ldots, n_k$. Then we have
    \[
        R(S_{n_1}, S_{n_2}, \ldots, S_{n_k}) = \begin{cases}
            \sum_{j = 1}^k (n_j - 2) + 1, & \mbox{if $t$ is even and positive},\\
            \sum_{j = 1}^k (n_j - 2) + 2, & \mbox{otherwise}.
        \end{cases}
    \]
\end{theorem}

\noindent
With Proposition \ref{prel_star} in mind, Theorem \ref{star_cool_th} determines all the cyclic Ramsey numbers for star graphs, while Proposition~\ref{prel_star_2} does the same for the ordered Ramsey numbers involving start-central stars.

We now turn our attention to the ordered and cyclic Ramsey numbers where one argument is a start-central star, while the other is a path or a cycle. We begin by investigating the cyclic Ramsey numbers of start-central stars versus monotone paths; see Table \ref{ssc_pmon_cyc_tab}. Our computational results suggest the following conjecture.

\begin{conjecture}\label{s_pmon_cyc_conj}
    For any $a \ge 3$ and $b \ge 4$, we have $R_\mathrm{cyc}(S_a, P_b^\mathrm{mon}) = 1 + (a - 1)(b - 2)$.
\end{conjecture}

\begin{table}[H]
\centering
\footnotesize
\begin{tabular}{|c||rrrrrrrrrrrrr|}
\hline
\backslashbox{$a$}{$b$} & $3$ & $4$ & $5$ & $6$ & $7$ & $8$ & $9$ & $10$ & $11$ & $12$ & $13$ & $14$ & $15$\\
\hline
\hline
$3$ & $3$ & $5$ & $7$ & $9$ & $11$ & $13$ & $15$ & $17$ & $19$ & $21$ & $23$ & $\ge 25$ & $\ge 26$\\
$4$ & $5$ & $7$ & $10$ & $13$ & $16$ & $19$ & $22$ & $25$ & $\ge 27$ & & & &\\
$5$ & $5$ & $9$ & $13$ & $17$ & $21$ & $25$ & $\ge 29$ & & & & & &\\
$6$ & $7$ & $11$ & $16$ & $21$ & $26$ & $\ge 31$ & & & & & & &\\
$7$ & $7$ & $13$ & $19$ & $25$ & $\ge 31$ & $\ge 33$ & & & & & & &\\
$8$ & $9$ & $15$ & $22$ & $29$ & $\ge 34$ & & & & & & & &\\
$9$ & $9$ & $17$ & $25$ & $\ge 30$ & & & & & & & & &\\
$10$ & $11$ & $19$ & $\ge 28$ & & & & & & & & & &\\
$11$ & $11$ & $21$ & & & & & & & & & & &\\
$12$ & $13$ & $23$ & & & & & & & & & & &\\
$13$ & $13$ & $25$ & & & & & & & & & & &\\
$14$ & $15$ & $\ge 26$ & & & & & & & & & & &\\
$15$ & $15$ & & & & & & & & & & & &\\
\hline
\end{tabular}
\caption{Cyclic Ramsey numbers $R_\mathrm{cyc}(S_a, P_b^\mathrm{mon})$ with $3 \le a, b \le 15$.}
\label{ssc_pmon_cyc_tab}
\end{table}

As before, the ordered Ramsey numbers involving alternating or reverse alternating paths do not seem to exhibit an easily discernible structure; see Tables \ref{ssc_palt_ord_tab} and \ref{ssc_pralt_ord_tab}. On the other hand, our computed values of cyclic Ramsey numbers of stars versus alternating paths, which are given in Table \ref{ssc_palt_cyc_tab}, lead to the following conjecture.

\begin{conjecture}\label{s_palt_cyc_conj}
    For any $a \ge 3$ and $b \ge 4$, we have $R_\mathrm{cyc}(S_a, P_b^\mathrm{alt}) = a + b - 2 - (ab \bmod 2)$.
\end{conjecture}

\begin{table}[H]
\centering
\footnotesize
\begin{tabular}{|c||rrrrrrrrrrrrr|}
\hline
\backslashbox{$a$}{$b$} & $3$ & $4$ & $5$ & $6$ & $7$ & $8$ & $9$ & $10$ & $11$ & $12$ & $13$ & $14$ & $15$\\
\hline
\hline
$3$ & $5$ & $6$ & $7$ & $8$ & $10$ & $11$ & $12$ & $13$ & $14$ & $15$ & $17$ & $18$ & $19$\\
$4$ & $6$ & $7$ & $9$ & $10$ & $12$ & $13$ & $14$ & $15$ & $17$ & $18$ & $19$ & $\ge 20$ & $\ge 21$\\
$5$ & $8$ & $9$ & $11$ & $12$ & $13$ & $14$ & $16$ & $17$ & $\ge 19$ & $\ge 20$ & $\ge 21$ & &\\
$6$ & $9$ & $10$ & $12$ & $13$ & $15$ & $16$ & $\ge 18$ & $\ge 19$ & $\ge 20$ & & & &\\
$7$ & $10$ & $11$ & $14$ & $15$ & $\ge 17$ & $\ge 18$ & $\ge 19$ & & & & & &\\
$8$ & $12$ & $13$ & $15$ & $16$ & $\ge 18$ & $\ge 19$ & $\ge 21$ & & & & & &\\
$9$ & $13$ & $14$ & $\ge 16$ & $\ge 17$ & $\ge 20$ & & & & & & & &\\
$10$ & $14$ & $15$ & $\ge 18$ & $\ge 19$ & $\ge 21$ & & & & & & & &\\
$11$ & $15$ & $16$ & $\ge 19$ & $\ge 20$ & $\ge 22$ & & & & & & & &\\
$12$ & $17$ & $18$ & $\ge 20$ & $\ge 21$ & $\ge 24$ & & & & & & & &\\
$13$ & $\ge 18$ & $\ge 19$ & $\ge 22$ & & & & & & & & & &\\
$14$ & $\ge 19$ & $\ge 20$ & $\ge 23$ & & & & & & & & & &\\
$15$ & $\ge 20$ & $\ge 21$ & $\ge 24$ & & & & & & & & & &\\
\hline
\end{tabular}
\caption{Ordered Ramsey numbers $R_\mathrm{ord}(S_a^\mathrm{sc}, P_b^\mathrm{alt})$ with $3 \le a, b \le 15$.}
\label{ssc_palt_ord_tab}
\end{table}

\begin{table}[H]
\centering
\footnotesize
\begin{tabular}{|c||rrrrrrrrrrrrr|}
\hline
\backslashbox{$a$}{$b$} & $3$ & $4$ & $5$ & $6$ & $7$ & $8$ & $9$ & $10$ & $11$ & $12$ & $13$ & $14$ & $15$\\
\hline
\hline
$3$ & $4$ & $6$ & $7$ & $8$ & $9$ & $11$ & $12$ & $13$ & $14$ & $15$ & $16$ & $18$ & $19$\\
$4$ & $5$ & $7$ & $8$ & $10$ & $11$ & $13$ & $14$ & $15$ & $16$ & $18$ & $19$ & $20$ & $\ge 21$\\
$5$ & $6$ & $9$ & $10$ & $12$ & $13$ & $14$ & $15$ & $17$ & $\ge 18$ & $\ge 20$ & $\ge 21$ & &\\
$6$ & $7$ & $10$ & $11$ & $13$ & $14$ & $16$ & $\ge 17$ & $\ge 19$ & $\ge 20$ & & & &\\
$7$ & $8$ & $11$ & $12$ & $15$ & $16$ & $\ge 18$ & $\ge 19$ & $\ge 20$ & & & & &\\
$8$ & $9$ & $13$ & $14$ & $16$ & $\ge 17$ & $\ge 19$ & $\ge 20$ & & & & & &\\
$9$ & $10$ & $14$ & $15$ & $\ge 17$ & $\ge 18$ & $\ge 21$ & & & & & & &\\
$10$ & $11$ & $15$ & $16$ & $\ge 19$ & $\ge 20$ & $\ge 22$ & & & & & & &\\
$11$ & $12$ & $16$ & $17$ & $\ge 20$ & $\ge 21$ & $\ge 23$ & & & & & & &\\
$12$ & $13$ & $18$ & $\ge 19$ & $\ge 21$ & $\ge 22$ & & & & & & & &\\
$13$ & $14$ & $\ge 19$ & $\ge 20$ & $\ge 23$ & & & & & & & & &\\
$14$ & $15$ & $\ge 20$ & $\ge 21$ & $\ge 24$ & & & & & & & & &\\
$15$ & $16$ & $\ge 21$ & $\ge 22$ & $\ge 25$ & & & & & & & & &\\
\hline
\end{tabular}
\caption{Ordered Ramsey numbers $R_\mathrm{ord}(S_a^\mathrm{sc}, P_b^\mathrm{ralt})$ with $3 \le a, b \le 15$.}
\label{ssc_pralt_ord_tab}
\end{table}

\begin{table}[H]
\centering
\footnotesize
\begin{tabular}{|c||rrrrrrrrrrrrr|}
\hline
\backslashbox{$a$}{$b$} & $3$ & $4$ & $5$ & $6$ & $7$ & $8$ & $9$ & $10$ & $11$ & $12$ & $13$ & $14$ & $15$\\
\hline
\hline
$3$ & $3$ & $5$ & $5$ & $7$ & $7$ & $9$ & $9$ & $11$ & $11$ & $13$ & $13$ & $15$ & $15$\\
$4$ & $5$ & $6$ & $7$ & $8$ & $9$ & $10$ & $11$ & $12$ & $13$ & $14$ & $15$ & $16$ & $\ge 17$\\
$5$ & $5$ & $7$ & $7$ & $9$ & $9$ & $11$ & $11$ & $13$ & $\ge 13$ & $\ge 15$ & $\ge 15$ & &\\
$6$ & $7$ & $8$ & $9$ & $10$ & $11$ & $12$ & $\ge 13$ & $\ge 14$ & $\ge 15$ & & & &\\
$7$ & $7$ & $9$ & $9$ & $11$ & $\ge 11$ & $\ge 13$ & $\ge 13$ & & & & & &\\
$8$ & $9$ & $10$ & $11$ & $12$ & $\ge 13$ & $\ge 14$ & $\ge 15$ & & & & & &\\
$9$ & $9$ & $11$ & $11$ & $\ge 13$ & $\ge 13$ & $\ge 15$ & & & & & & &\\
$10$ & $11$ & $12$ & $13$ & $\ge 14$ & $\ge 15$ & $\ge 16$ & & & & & & &\\
$11$ & $11$ & $13$ & $\ge 13$ & $\ge 15$ & $\ge 15$ & & & & & & & &\\
$12$ & $13$ & $14$ & $\ge 15$ & $\ge 16$ & $\ge 17$ & & & & & & & &\\
$13$ & $13$ & $15$ & $\ge 15$ & $\ge 17$ & $\ge 17$ & & & & & & & &\\
$14$ & $15$ & $16$ & $\ge 17$ & $\ge 18$ & $\ge 19$ & & & & & & & &\\
$15$ & $15$ & $17$ & $\ge 17$ & $\ge 19$ & $\ge 19$ & & & & & & & &\\
\hline
\end{tabular}
\caption{Cyclic Ramsey numbers $R_\mathrm{cyc}(S_a, P_b^\mathrm{alt})$ with $3 \le a, b \le 15$.}
\label{ssc_palt_cyc_tab}
\end{table}

In other words, the cyclic Ramsey numbers of stars versus alternating paths seem to be the same as those of alternating paths versus alternating paths, or alternating paths versus reverse alternating paths. We mention in passing that, by a reflection argument, $R_\mathrm{cyc}(S_a, P_b^\mathrm{alt}) = R_\mathrm{cyc}(S_a, P_b^\mathrm{ralt})$ holds for any $a, b \in \mathbb{N}$.

Now, we combine start-central stars with monotone cycles and consider the corresponding ordered and cyclic Ramsey numbers. Instead of presenting tables with computed values, we directly determine all of these Ramsey numbers.

\begin{theorem}\label{ssc_cmon_ord_th}
    For any $a \in \mathbb{N}$ and $b \ge 2$, we have $R_\mathrm{ord}(S_a^\mathrm{sc}, C_b^\mathrm{mon}) = 1 + (a - 1)(b - 1)$.
\end{theorem}
\begin{proof}
    Since $P_b^\mathrm{mon}$ is a subgraph of $C_b^\mathrm{mon}$, Corollary \ref{star_pmon_cor} implies
    \[
        R_\mathrm{ord}(S_a^\mathrm{sc}, C_b^\mathrm{mon}) \ge R_\mathrm{ord}(S_a^\mathrm{sc}, P_b^\mathrm{mon}) = 1 + (a - 1)(b - 1) .
    \]
    Thus, it suffices to prove that any $2$-edge-coloring of $K_n$ with $n = 1 + (a - 1)(b - 1)$ contains $S_a^\mathrm{sc}$ as a monochromatic subgraph of color $1$ or $C_b^\mathrm{mon}$ as a monochromatic subgraph of color $2$, with the vertices required to appear in the same order.

     Let $G$ be a $2$-edge-coloring of $K_n$ that does not contain $S_a^\mathrm{sc}$ as a forbidden subgraph. Now, let $G_1$ be the subgraph of $G$ induced on the vertices $v > 0$ such that the edge $\{ 0, v \}$ is of color $2$. Observe that $|G_1| \ge n - (a - 1) = 1 + (a - 1)(b - 2)$. Since $G_1$ does not contain $S_a^\mathrm{sc}$ as a monochromatic subgraph of color $1$ with the vertices appearing in the same order, Corollary \ref{star_pmon_cor} implies that $G_1$ contains $P_{b - 1}^\mathrm{mon}$ as a monochromatic subgraph of color $2$ with a consistent vertex ordering. By attaching the vertex $0$ to the endpoints of this path in the original graph $G$, we obtain a forbidden monochromatic subgraph $C_b^\mathrm{mon}$.
\end{proof}

\noindent
The following result is an immediate corollary of Proposition \ref{cyc_basic_prop}, Theorem \ref{one_more_th} and Theorem \ref{ssc_cmon_ord_th}.

\begin{corollary}
    For any $a \in \mathbb{N}$ and $b \ge 2$, we have $R_\mathrm{cyc}(S_a, C_b^\mathrm{mon}) = 1 + (a - 1)(b - 1)$.
\end{corollary}

\subsection{Complete graphs versus paths, cycles and stars}

By Proposition \ref{cyc_basic_prop}, when all arguments are complete graphs, the cyclic Ramsey numbers coincide with the ordered and the standard Ramsey numbers. For further results on these numbers, the reader can refer to the survey \cite{Radziszowski} by Radziszowski and references therein.

In the present subsection, we study the ordered and cyclic Ramsey numbers where one of the arguments is a complete graph, while the other is a path, a cycle, or a star. A well-known result by Parsons \cite{Parsons1973} determines all the standard Ramsey numbers of complete graphs versus paths.

\begin{theorem}[\hspace{1sp}{\cite{Parsons1973}}]
    For any $a, b \in \mathbb{N}$, we have $R(K_a, P_b) = 1 + (a - 1)(b - 1)$.
\end{theorem}

\noindent
The above theorem, together with Theorem \ref{pmon_k_th} and Proposition \ref{cyc_basic_prop}, yields the following corollary.

\begin{corollary}\label{k_pmon_cyc_cor}
    For any $a, b \in \mathbb{N}$, we have $R_\mathrm{cyc}(K_a, P_b^\mathrm{mon}) = 1 + (a - 1)(b - 1)$.
\end{corollary}

\noindent
In other words, the standard, ordered and cyclic Ramsey numbers all coincide when one of the arguments is a complete graph and the other is a monotone path.

Recall that the ordered and cyclic Ramsey numbers of monotone paths appear to be larger than the corresponding numbers of alternating paths. If we compare these two types of paths in combination with the complete graphs, our computational results indicate that the situation seems to be different; see Tables~\ref{k_palt_ord_tab} and~\ref{k_palt_cyc_tab}. As it turns out, alternating paths appear to yield identical cyclic Ramsey numbers and larger ordered Ramsey numbers. Moreover, the results suggest the following conjectures.

\begin{conjecture}\label{k_palt_ord_conj}
    For any $a \ge 2$ and $b \ge 3$, we have $R_\mathrm{ord}(K_a, P_b^\mathrm{alt}) = \lceil \frac{3ab - 5a}{2} \rceil - 2b + 5$.
\end{conjecture}

\begin{conjecture}\label{k_palt_cyc_conj}
    For any $a, b \in \mathbb{N}$, we have $R_\mathrm{cyc}(K_a, P_b^\mathrm{alt}) = 1 + (a - 1)(b - 1)$.
\end{conjecture}

\begin{table}[H]
\centering
\footnotesize
\begin{tabular}{|c||rrrrrr|}
\hline
\backslashbox{$a$}{$b$} & $3$ & $4$ & $5$ & $6$ & $7$ & $8$\\
\hline
\hline
$3$ & $5$ & $8$ & $10$ & $13$ & $15$ & $18$\\
$4$ & $7$ & $11$ & $15$ & $\ge 19$ & $\ge 22$ & $\ge 26$\\
$5$ & $9$ & $15$ & $\ge 20$ & $\ge 25$ & $\ge 30$ &\\
$6$ & $11$ & $\ge 18$ & $\ge 25$ & $\ge 31$ & &\\
$7$ & $13$ & $\ge 22$ & $\ge 29$ & $\ge 36$ & &\\
$8$ & $15$ & $\ge 25$ & $\ge 33$ & $\ge 37$ & &\\
\hline
\end{tabular}
\caption{Ordered Ramsey numbers $R_\mathrm{ord}(K_a, P_b^\mathrm{alt})$ with $3 \le a, b \le 8$.}
\label{k_palt_ord_tab}
\end{table}

\begin{table}[H]
\centering
\footnotesize
\begin{tabular}{|c||rrrrrr|}
\hline
\backslashbox{$a$}{$b$} & $3$ & $4$ & $5$ & $6$ & $7$ & $8$\\
\hline
\hline
$3$ & $5$ & $7$ & $9$ & $11$ & $13$ & $15$\\
$4$ & $7$ & $10$ & $13$ & $\ge 16$ & $\ge 19$ & $\ge 22$\\
$5$ & $9$ & $13$ & $\ge 17$ & $\ge 21$ & $\ge 25$ &\\
$6$ & $11$ & $16$ & $\ge 21$ & $\ge 26$ & $\ge 31$ &\\
$7$ & $13$ & $\ge 19$ & $\ge 25$ & $\ge 31$ & &\\
$8$ & $15$ & $\ge 22$ & $\ge 29$ & & &\\
\hline
\end{tabular}
\caption{Cyclic Ramsey numbers $R_\mathrm{cyc}(K_a, P_b^\mathrm{alt})$ with $3 \le a, b \le 8$.}
\label{k_palt_cyc_tab}
\end{table}

Next, we study the ordered and cyclic Ramsey numbers of complete graphs versus monotone cycles. Since any complete graph or monotone cycle is rotationally isomorphic only to itself, it follows that $R_\mathrm{ord}(K_a, C_b^\mathrm{mon}) = R_\mathrm{cyc}(K_a, C_b^\mathrm{mon})$ for any $a \in \mathbb{N}$ and $b \ge 2$, i.e., the two numbers coincide for these pairs of graphs. The computed values are reported in Table \ref{k_cmon_ord_tab} and indicate very rapid growth.

\begin{table}[H]
\centering
\footnotesize
\begin{tabular}{|c||rrrrrr|}
\hline
\backslashbox{$a$}{$b$} & $3$ & $4$ & $5$ & $6$ & $7$ & $8$\\
\hline
\hline
$3$ & $6$ & $9$ & $12$ & $15$ & $18$ & $\ge 21$\\
$4$ & $9$ & $\ge 14$ & $\ge 19$ & $\ge 24$ & &\\
$5$ & $\ge 14$ & $\ge 23$ & $\ge 32$ & & &\\
$6$ & $\ge 18$ & $\ge 30$ & $\ge 39$ & & &\\
$7$ & $\ge 23$ & $\ge 36$ & $\ge 44$ & & &\\
$8$ & $\ge 28$ & $\ge 37$ & & & &\\
\hline
\end{tabular}
\caption{Ordered and cyclic Ramsey numbers $R_\mathrm{ord}(K_a, C_b^\mathrm{mon}) = R_\mathrm{cyc}(K_a, C_b^\mathrm{mon})$ with $3 \le a, b \le 8$.}
\label{k_cmon_ord_tab}
\end{table}

In the rest of the subsection, we combine complete graphs with start-central stars. We proceed by finding all the ordered Ramsey numbers of complete graphs versus start-central stars.

\begin{theorem}\label{k_ssc_ord_th}
    For any $a, b \in \mathbb{N}$, we have $R_\mathrm{ord}(K_a, S_b^\mathrm{sc}) = 1 + (a - 1)(b - 1)$.
\end{theorem}
\begin{proof}
    The theorem trivially holds if $\min \{ a, b \} = 1$, so we assume that $a, b \ge 2$. Since $P_a^\mathrm{mon}$ is a subgraph of $K_a$, Corollary \ref{star_pmon_cor} implies
    \[
        R_\mathrm{ord}(K_a, S_b^\mathrm{sc}) \ge R_\mathrm{ord}(P_a^\mathrm{mon}, S_b^\mathrm{sc}) = 1 + (a - 1)(b - 1).
    \]
    Therefore, to complete the proof, it suffices to show that any $2$-edge-coloring of $K_n$ with $n = 1 + (a - 1)(b - 1)$ contains $K_a$ as a monochromatic subgraph of color $1$ or $S_b^\mathrm{sc}$ as a monochromatic subgraph of color $2$, with the vertices appearing in the same order.

    We carry out the proof by induction on $a$ using a similar approach to Theorem \ref{ssc_cmon_ord_th}. To begin, the statement trivially holds for $a = 2$, so we assume that $a \ge 3$. Let $G$ be a $2$-edge-coloring of $K_n$ that does not contain $S_b^\mathrm{sc}$ as a forbidden subgraph. Now, let $G_1$ be the subgraph of $G$ induced on the vertices $v > 0$ such that the edge $\{ 0, v \}$ is of color $1$. Observe that $|G_1| \ge n - (b - 1) = 1 + (a - 2)(b - 1)$. Since $G_1$ does not contain $S_b^\mathrm{sc}$ as a monochromatic subgraph of color $2$ with the vertices appearing in the same order, by induction, we have that $G_1$ contains $K_{a - 1}$ as a monochromatic subgraph of color $1$ with a consistent vertex ordering. By attaching the vertex $0$ to all the vertices of this copy of $K_{a - 1}$ in the original graph $G$, we obtain a forbidden monochromatic subgraph $K_a$.
\end{proof}

\noindent
Recall the following well-known result shown by Chvátal \cite{Chvatal1977}.

\begin{theorem}[\hspace{1sp}{\cite{Chvatal1977}}]
    For any $a, b \in \mathbb{N}$ and tree $T$ of order $a$, we have $R(T, K_b) = 1 + (a - 1)(b - 1)$.
\end{theorem}

\noindent
The standard and cyclic Ramsey numbers of complete graphs versus stars clearly coincide, hence we obtain the following result.

\begin{corollary}\label{k_ssc_cyc_th}
    For any $a, b \in \mathbb{N}$, we have $R_\mathrm{cyc}(K_a, S_b) = 1 + (a - 1)(b - 1)$.
\end{corollary}

\noindent
We mention in passing that Corollary~\ref{k_ssc_cyc_th} can also be proved through Proposition \ref{cyc_basic_prop}, Theorem \ref{one_more_th} and Theorem~\ref{k_ssc_ord_th}. Therefore, the standard, ordered and cyclic Ramsey numbers of complete graphs versus start-central stars are identical.

\subsection{Nested matchings versus all classes}

In the last subsection, we deal with nested matchings and investigate the ordered and cyclic Ramsey numbers where at least one argument is a nested matching. To begin, we combine nested matchings with themselves, in which case the ordered Ramsey numbers can be explicitly determined.

\begin{theorem}\label{mnest_mnest_ord_th}
    For any even $a, b \ge 2$, we have $R_\mathrm{ord}(M_a^\mathrm{nest}, M_b^\mathrm{nest}) = a + b - 2$.
\end{theorem}
\begin{proof}
    First, we show that $R_\mathrm{ord}(M_a^\mathrm{nest}, M_b^\mathrm{nest}) \ge a + b - 2$ by constructing a $2$-edge-coloring of $K_n$ with $n = a + b - 3$ that contains neither $M_a^\mathrm{nest}$ as a monochromatic subgraph of color $1$ nor $M_b^\mathrm{nest}$ as a monochromatic subgraph of color $2$, with a consistent vertex ordering.
    
    Let $G$ be the $2$-edge-coloring of $K_n$ such that an edge between any two distinct vertices $u, v \in \{ 0, 1, 2, \ldots, \linebreak n - 1 \}$ is colored with color $2$ if and only if $\frac{a}{2} - 1 \le u, v \le n - \frac{a}{2}$. If $G$ contains a forbidden subgraph $M_b^\mathrm{nest}$ of color~$2$, then all of its vertices lie in $\{ \frac{a}{2} - 1, \frac{a}{2}, \ldots, n - \frac{a}{2} \}$ since the remaining vertices are incident only to edges of color $1$. However, 
    the set $\{ \frac{a}{2} - 1, \frac{a}{2}, \ldots, n - \frac{a}{2} \}$ contains $n - a + 2 = b - 1 < b$ elements, yielding a contradiction.

    Now, suppose that $G$ contains a forbidden subgraph $M_a^\mathrm{nest}$ of color~$1$. Let $u_1 v_1, u_2 v_2, \ldots, u_\frac{a}{2} v_\frac{a}{2}$ be the edges of this copy of $M_a^\mathrm{nest}$, so that
    \[
        u_1 < u_2 < \cdots < u_\frac{a}{2} < v_\frac{a}{2} < \cdots < v_2 < v_1 .
    \]
    Observe that $u_\frac{a}{2} \ge \frac{a}{2} - 1$ and $v_\frac{a}{2} \le n - \frac{a}{2}$. This implies that the edge $u_\frac{a}{2} v_\frac{a}{2}$ is of color $2$, which is impossible. Therefore, $R_\mathrm{ord}(M_a^\mathrm{nest}, M_b^\mathrm{nest}) \ge a + b - 2$.

    We complete the proof by showing that any $2$-edge-coloring of $K_n$ with $n = a + b - 2$ contains $M_a^\mathrm{nest}$ as a monochromatic subgraph of color~$1$ or $M_b^\mathrm{nest}$ as a monochromatic subgraph of color~$2$, with a consistent vertex ordering. Consider the edges of the form $\{ j, n - 1 - j \}$ with $j \in \{ 0, 1, 2, \ldots, \frac{n}{2} - 1 \}$. Since there are $\frac{n}{2} = \frac{a}{2} + \frac{b}{2} - 1$ of them, it follows that at least $\frac{a}{2}$ of these edges are of color~$1$ or at least $\frac{b}{2}$ of them are of color~$2$. This implies that the given $2$-edge-coloring of $K_n$ contains a forbidden subgraph, hence $R_\mathrm{ord}(M_a^\mathrm{nest}, M_b^\mathrm{nest}) \le a + b - 2$.
\end{proof}

The computational results for the cyclic Ramsey numbers of nested matchings are shown in Table \ref{mnest_mnest_cyc_tab}. Motivated by the computed values, we partially determine these numbers. To this end, we need the following lemma on the largest monochromatic nested matchings in complete graphs with a circulant edge coloring.

\begin{lemma}\label{circulant_lemma}
    Let $n, k \in \mathbb{N}$ and let $G$ be a $k$-edge-coloring of $K_n$ such that the automorphism $\begin{pmatrix}\begin{smallmatrix}
    0 & 1 & 2 & \cdots & n - 2 & n - 1\\
    1 & 2 & 3 & \cdots & n - 1 & 0
    \end{smallmatrix}\end{pmatrix}$ is edge-color-preserving. Then, for any selected $c \in \{ 1, 2, \ldots, k \}$, a largest monochromatic nested matching $M$ of color $c$ contained in $G$, with the vertices appearing in cyclic order, can be constructed using the following algorithm:
    \begin{algorithmic}[1]
        \State $v_0 \gets n$, $j \gets 1$, $M \gets \varnothing$
        \While{true}
            \State $u_j \gets j - 1$
            \State $W = \{ w : \mbox{$u_j < w < v_{j - 1}$ and $\{ w, u_j \}$ is of color $c$} \}$
            \If{$W = \varnothing$}
                \State \textbf{break}
            \EndIf
            \State $v_j  \gets \max W$
            \State $M \gets M \cup \{ u_j v_j \}$
            \State $j \gets j + 1$
        \EndWhile
    \end{algorithmic}
\end{lemma}
\begin{proof}
    Observe that, by construction, $M$ is indeed a monochromatic nested matching of color $c$ contained in $G$ with a consistent vertex cyclic ordering. Let $M'$ be a largest monochromatic nested matching of color $c$ contained in $G$, with the vertices appearing in cyclic order. We carry out the proof by showing that $M'$ can be transformed without changing the number of edges into another monochromatic nested matching of color $c$ that is contained in $M$.

    If no edge in $G$ is of color $c$, we are done, so we assume that at least one edge is of color $c$. Since $G$ has a circulant edge coloring, the edge colors in $M'$ do not change regardless of how we rotate $M'$. Therefore, we can assume without loss of generality that the edges of $M'$ are $u_1' v_1', u_2' v_2', \ldots, u_p' v_p'$ for some $p \in \mathbb{N}$, so that
    \[
        0 = u_1' < u_2' < \cdots < u_p' < v_p' < \cdots < v_2' < v_1' .
    \]
    Note that $u_1' = u_1$. In addition, we can replace $v_1'$ with the largest vertex whose edge connecting it to $u_1 = 0$ is of color $c$, without disrupting the nested matching structure. Thus, without loss of generality, we can write $v_1' = v_1$.

    After these initial transformations, the following two-step procedure should be applied iteratively for each $j = 2, 3, \ldots, p$. If $u_j' \neq u_j$, then all the vertices $u_j', u_{j + 1}', \ldots, u_p', v_p', \ldots, v_{j + 1}', v_j'$ should be rotated to the left so that $u_j'$ becomes $u_j = u_{j - 1} + 1$. Note that this operation does not break the nested matching structure, and since the edge coloring of $G$ is circulant, the edge colors also do not change. Subsequently, the newly obtained vertex $v_j'$ should be replaced with the largest vertex below $v_{j - 1}$ whose edge connecting it to $u_j$ is of color $c$. This again keeps the nested matching structure intact. After executing all the described steps, the nested matching $M'$ transforms into a subgraph of $M$, which completes the proof.
\end{proof}

\begin{table}[H]
\centering
\footnotesize
\begin{tabular}{|c||rrrrrrr|}
\hline
\backslashbox{$a$}{$b$} & $4$ & $6$ & $8$ & $10$ & $12$ & $14$ & $16$\\
\hline
\hline
$4$ & $5$ & $8$ & $9$ & $12$ & $13$ & $16$ & $17$\\
$6$ & & $10$ & $12$ & $14$ & $16$ & $18$ & $20$\\
$8$ & & & $13$ & $16$ & $17$ & $20$ & $21$\\
$10$ & & & & $18$ & $20$ & $22$ & $24$\\
$12$ & & & & & $21$ & $24$ & $\ge 25$\\
$14$ & & & & & & $\ge 26$ & $\ge 27$\\
\hline
\end{tabular}
\caption{Cyclic Ramsey numbers $R_\mathrm{cyc}(M_a^\mathrm{nest}, M_b^\mathrm{nest})$ with even $a, b$, alongside $4 \le a \le b$, $a \le 14$ and $b \le 16$.}
\label{mnest_mnest_cyc_tab}
\end{table}

\noindent
We are now in a position to partially determine the cyclic Ramsey numbers of nested matchings.

\begin{theorem}\label{mnest_mnest_cyc_thp}
    Let $a, b \ge 2$ be even, with at least one of them not divisible by four. Then $R_\mathrm{cyc}(M_a^\mathrm{nest}, M_b^\mathrm{nest}) = a + b - 2$.
\end{theorem}
\begin{proof}    
    By Proposition \ref{cyc_basic_prop} and Theorem \ref{mnest_mnest_ord_th}, we have $R_\mathrm{cyc}(M_a^\mathrm{nest}, M_b^\mathrm{nest}) \le a + b - 2$. Therefore, it suffices to find a $2$-edge-coloring of $K_n$ with $n = a + b - 3$ that contains neither $M_a^\mathrm{nest}$ as a monochromatic subgraph of color~$1$ nor $M_b^\mathrm{nest}$ as a monochromatic subgraph of color~$2$, with the vertices appearing in the same cyclic order. Without loss of generality, we assume that $a \equiv 2 \pmod 4$. Moreover, since the result is trivial for $\min \{ a, b \} = 2$, we assume that $a \ge 6$ and $b \ge 4$.

    Let $G$ be the $2$-edge-coloring of $K_n$ such that the edge between any two distinct vertices $u, v \in \{ 0, 1, 2, \ldots, \linebreak n - 1 \}$ is colored with color $1$ if and only if their circular distance is at most $\frac{a}{2} - 1$, i.e., $|v - u| \le \frac{a}{2} - 1$ or $|v - u| \ge n - \left( \frac{a}{2} - 1 \right)$. We prove that $G$ does not contain a forbidden subgraph using Lemma \ref{circulant_lemma}. For $c = 2$, Lemma~\ref{circulant_lemma} asserts that the nested matching $M$ with the edges $u_1 v_1, u_2 v_2, \ldots, u_{\frac{b}{2} - 1} v_{\frac{b}{2} - 1}$ defined by
    \[
        u_j = j - 1 \quad \mbox{and} \quad v_j = n - \tfrac{a}{2} + 1 - j \quad (j \in \{ 1, 2, \ldots, \tfrac{b}{2} - 1 \}),
    \]
    is a largest monochromatic nested matching of color $2$ contained in $G$, with the vertices in cyclic order. Note that all the edges connecting the vertices between $u_{\frac{b}{2} - 1} + 1 = \frac{b}{2} - 1$ and $v_{\frac{b}{2} - 1} - 1 = \frac{a}{2} + \frac{b}{2} - 2$ are of color~$1$. Since $M$ has $\frac{b}{2} - 1$ edges, this implies that $G$ does not contain $M_b^\mathrm{nest}$ as a monochromatic subgraph of color~$2$, with the vertices appearing in cyclic order.

    Now, we apply Lemma \ref{circulant_lemma} for $c = 1$. In this case, we get that the nested matching $M$ with the edges $u_1 v_1, \linebreak u_2 v_2, \ldots, u_{\frac{a}{2} - 1} v_{\frac{a}{2} - 1}$ defined by
    \[
        u_j = j - 1 \quad \mbox{and} \quad v_j = n - j \quad (j \in \{ 1, 2, \ldots, \tfrac{a - 2}{4} \}),
    \]
    and
    \[
        u_j = j - 1 \quad \mbox{and} \quad v_j = a - 1 - j \quad (j \in \{ \tfrac{a + 2}{4}, \tfrac{a + 6}{4}, \ldots, \tfrac{a}{2} - 1 \}),
    \]
    is a largest monochromatic nested matching of color $1$ contained in $G$, with the vertices in cyclic order. Since this matching has $\frac{a}{2} - 1$ edges, we conclude that $G$ does not contain $M_a^\mathrm{nest}$ as a monochromatic subgraph of color~$1$, with the vertices appearing in cyclic order.
\end{proof}

By employing the idea from the proof of Theorem \ref{mnest_mnest_cyc_thp}, and by setting the target graph order to $n = a + b - 5$ and changing the circular distance edge coloring threshold from $\frac{a}{2} - 1$ to $\frac{a}{2} - 2$, one can prove that $R_\mathrm{cyc}(M_a^\mathrm{nest}, M_b^\mathrm{nest}) \ge a + b - 4$ when $4 \mid a, b$. However, the results given in Table \ref{mnest_mnest_cyc_tab} indicate that this lower bound is not optimal. Instead, the computational results yield the following conjecture.

\begin{conjecture}\label{mnest_mnest_cyc_conj}
    Let $a, b \ge 4$ be divisible by four. Then $R_\mathrm{cyc}(M_a^\mathrm{nest}, M_b^\mathrm{nest}) = a + b - 3$.
\end{conjecture}

Next, we combine nested matchings with monotone paths and consider the corresponding ordered and cyclic Ramsey numbers. Our computational results are shown in Tables \ref{mnest_pmon_ord_tab} and \ref{mnest_pmon_cyc_tab} and suggest the following conjectures.

\begin{conjecture}
    For any even $a \ge 2$ and any $b \in \mathbb{N}$, we have $R_\mathrm{ord}(M_a^\mathrm{nest}, P_b^\mathrm{mon}) = 1 + (a - 1)(b - 1)$.
\end{conjecture}
\begin{conjecture}
    For any even $a \ge 2$ and any $b \ge 3$, we have
    \[
        R_\mathrm{cyc}(M_a^\mathrm{nest}, P_b^\mathrm{mon}) = \begin{cases}
            \frac{ab}{2} - \frac{a}{2} - b + 3, & \mbox{if $a \equiv 0 \pmod 4$},\\
            \frac{ab}{2} - \frac{a}{2} + 1, & \mbox{if $a \equiv 2 \pmod 4$}.
        \end{cases}
    \]
\end{conjecture}

\begin{table}[H]
\centering
\footnotesize
\begin{tabular}{|c||rrrrrrrrr|}
\hline
\backslashbox{$a$}{$b$} & $3$ & $4$ & $5$ & $6$ & $7$ & $8$ & $9$ & $10$ & $11$\\
\hline
\hline
$4$ & $7$ & $10$ & $13$ & $16$ & $19$ & $22$ & $25$ & $28$ & $\ge 31$\\
$6$ & $11$ & $16$ & $21$ & $26$ & $31$ & $\ge 36$ & $\ge 39$ & &\\
$8$ & $15$ & $22$ & $29$ & $\ge 36$ & $\ge 43$ & & & &\\
$10$ & $19$ & $28$ & $\ge 37$ & & & & & &\\
$12$ & $23$ & $\ge 32$ & & & & & & &\\
$14$ & $27$ & & & & & & & &\\
$16$ & $\ge 31$ & & & & & & & &\\
\hline
\end{tabular}
\caption{Ordered Ramsey numbers $R_\mathrm{ord}(M_a^\mathrm{nest}, P_b^\mathrm{mon})$ with even $4 \le a \le 16$ and any $3 \le b \le 11$.}
\label{mnest_pmon_ord_tab}
\end{table}

\begin{table}[H]
\centering
\footnotesize
\begin{tabular}{|c||rrrrrrrrrrrrrrrrrr|}
\hline
\backslashbox{$a$}{$b$} & $3$ & $4$ & $5$ & $6$ & $7$ & $8$ & $9$ & $10$ & $11$ & $12$ & $13$ & $14$ & $15$ & $16$ & $17$ & $18$ & $19$ & $20$\\
\hline
\hline
$4$ & $4$ & $5$ & $6$ & $7$ & $8$ & $9$ & $10$ & $11$ & $12$ & $13$ & $14$ & $15$ & $16$ & $17$ & $18$ & $19$ & $20$ & $21$\\
$6$ & $7$ & $10$ & $13$ & $16$ & $19$ & $22$ & $25$ & $\ge 28$ & & & & & & & & & &\\
$8$ & $8$ & $11$ & $14$ & $17$ & $20$ & $23$ & $\ge 26$ & $\ge 29$ & & & & & & & & & &\\
$10$ & $11$ & $16$ & $21$ & $\ge 23$ & $\ge 28$ & & & & & & & & & & & & &\\
$12$ & $12$ & $17$ & $\ge 22$ & $\ge 26$ & & & & & & & & & & & & & &\\
$14$ & $15$ & $22$ & $\ge 25$ & $\ge 27$ & & & & & & & & & & & & & &\\
$16$ & $16$ & $\ge 23$ & $\ge 26$ & & & & & & & & & & & & & & &\\
\hline
\end{tabular}
\caption{Cyclic Ramsey numbers $R_\mathrm{cyc}(M_a^\mathrm{nest}, P_b^\mathrm{mon})$ with even $4 \le a \le 16$ and any $3 \le b \le 20$.}
\label{mnest_pmon_cyc_tab}
\end{table}

If we replace the monotone paths with the alternating paths, then the 
growth of the ordered and cyclic Ramsey numbers appears to be linear; see Tables \ref{mnest_palt_ord_tab} and \ref{mnest_palt_cyc_tab}. Despite this, both of these Ramsey numbers do not seem to exhibit an easily discernible structure. Similarly, the ordered and cyclic Ramsey numbers of nested matchings versus monotone cycles seem to follow a pattern that is difficult to decipher; see Tables \ref{mnest_cmon_ord_tab} and \ref{mnest_cmon_cyc_tab}. For further results on the particular case of $R_\mathrm{ord}(M_a^\mathrm{nest}, C_3^\mathrm{mon})$, i.e., $R_\mathrm{ord}(M_a^\mathrm{nest}, K_3)$, the reader can refer to the papers \cite{BalPolj2023} and \cite{BalPolj2024} by Balko and Poljak.

\begin{table}[H]
\centering
\footnotesize
\resizebox{\textwidth}{!}{
\begin{tabular}{|c||rrrrrrrrrrrrrrrrrr|}
\hline
\backslashbox{$a$}{$b$} & $3$ & $4$ & $5$ & $6$ & $7$ & $8$ & $9$ & $10$ & $11$ & $12$ & $13$ & $14$ & $15$ & $16$ & $17$ & $18$ & $19$ & $20$\\
\hline
\hline
$4$ & $6$ & $7$ & $8$ & $9$ & $10$ & $11$ & $12$ & $13$ & $14$ & $15$ & $16$ & $17$ & $18$ & $19$ & $20$ & $21$ & $22$ & $23$\\
$6$ & $8$ & $10$ & $11$ & $12$ & $13$ & $14$ & $15$ & $16$ & $17$ & $18$ & $19$ & $20$ & $21$ & $22$ & $23$ & $24$ & $25$ & $26$\\
$8$ & $11$ & $12$ & $14$ & $15$ & $16$ & $17$ & $18$ & $19$ & $20$ & $21$ & $22$ & $23$ & $24$ & $\ge 25$ & $\ge 26$ & $\ge 27$ & &\\
$10$ & $13$ & $15$ & $17$ & $18$ & $19$ & $20$ & $21$ & $22$ & $\ge 23$ & $\ge 24$ & $\ge 25$ & & & & & & &\\
$12$ & $15$ & $17$ & $19$ & $21$ & $22$ & $\ge 23$ & $\ge 24$ & $\ge 25$ & & & & & & & & & &\\
$14$ & $18$ & $20$ & $22$ & $\ge 23$ & $\ge 25$ & $\ge 25$ & & & & & & & & & & & &\\
$16$ & $20$ & $22$ & $\ge 24$ & $\ge 26$ & $\ge 27$ & & & & & & & & & & & & &\\
\hline
\end{tabular}
}
\caption{Ordered Ramsey numbers $R_\mathrm{ord}(M_a^\mathrm{nest}, P_b^\mathrm{alt})$ with even $4 \le a \le 16$ and any $3 \le b \le 20$.}
\label{mnest_palt_ord_tab}
\end{table}

\begin{table}[H]
\centering
\footnotesize
\resizebox{\textwidth}{!}{
\begin{tabular}{|c||rrrrrrrrrrrrrrrrrr|}
\hline
\backslashbox{$a$}{$b$} & $3$ & $4$ & $5$ & $6$ & $7$ & $8$ & $9$ & $10$ & $11$ & $12$ & $13$ & $14$ & $15$ & $16$ & $17$ & $18$ & $19$ & $20$\\
\hline
\hline
$4$ & $4$ & $6$ & $6$ & $8$ & $8$ & $10$ & $10$ & $12$ & $12$ & $14$ & $14$ & $16$ & $16$ & $18$ & $18$ & $20$ & $20$ & $22$\\
$6$ & $7$ & $8$ & $9$ & $10$ & $11$ & $12$ & $13$ & $14$ & $15$ & $16$ & $17$ & $18$ & $19$ & $20$ & $21$ & $22$ & $23$ & $24$\\
$8$ & $8$ & $10$ & $10$ & $12$ & $12$ & $14$ & $14$ & $16$ & $16$ & $18$ & $18$ & $20$ & $\ge 20$ & $\ge 22$ & & & &\\
$10$ & $11$ & $12$ & $13$ & $14$ & $15$ & $16$ & $17$ & $\ge 18$ & $\ge 19$ & $\ge 20$ & & & & & & & &\\
$12$ & $12$ & $14$ & $15$ & $\ge 16$ & $\ge 17$ & $\ge 18$ & & & & & & & & & & & &\\
$14$ & $15$ & $16$ & $17$ & $\ge 18$ & $\ge 19$ & $\ge 20$ & & & & & & & & & & & &\\
$16$ & $16$ & $18$ & $\ge 19$ & $\ge 20$ & $\ge 21$ & & & & & & & & & & & & &\\
\hline
\end{tabular}
}
\caption{Cyclic Ramsey numbers $R_\mathrm{cyc}(M_a^\mathrm{nest}, P_b^\mathrm{alt})$ with even $4 \le a \le 16$ and any $3 \le b \le 20$.}
\label{mnest_palt_cyc_tab}
\end{table}

\begin{table}[H]
\centering
\footnotesize
\begin{tabular}{|c||rrrrrrrrr|}
\hline
\backslashbox{$a$}{$b$} & $3$ & $4$ & $5$ & $6$ & $7$ & $8$ & $9$ & $10$ & $11$\\
\hline
\hline
$4$ & $7$ & $10$ & $13$ & $16$ & $19$ & $22$ & $25$ & $28$ & $\ge 31$\\
$6$ & $11$ & $16$ & $21$ & $26$ & $31$ & $\ge 36$ & $\ge 36$ & &\\
$8$ & $16$ & $22$ & $29$ & $\ge 36$ & $\ge 43$ & $\ge 40$ & & &\\
$10$ & $20$ & $\ge 28$ & $\ge 35$ & & & & & &\\
$12$ & $\ge 25$ & $\ge 31$ & & & & & & &\\
$14$ & $\ge 27$ & & & & & & & &\\
$16$ & $\ge 30$ & & & & & & & &\\
\hline
\end{tabular}
\caption{Ordered Ramsey numbers $R_\mathrm{ord}(M_a^\mathrm{nest}, C_b^\mathrm{mon})$ with even $4 \le a \le 16$ and any $3 \le b \le 11$.}
\label{mnest_cmon_ord_tab}
\end{table}

\begin{table}[H]
\centering
\footnotesize
\begin{tabular}{|c||rrrrrrrrrrrrrrrrrr|}
\hline
\backslashbox{$a$}{$b$} & $3$ & $4$ & $5$ & $6$ & $7$ & $8$ & $9$ & $10$ & $11$ & $12$ & $13$ & $14$ & $15$ & $16$ & $17$ & $18$ & $19$ & $20$\\
\hline
\hline
$4$ & $6$ & $6$ & $7$ & $8$ & $9$ & $10$ & $11$ & $12$ & $13$ & $14$ & $15$ & $16$ & $17$ & $18$ & $19$ & $20$ & $21$ & $22$\\
$6$ & $9$ & $12$ & $15$ & $18$ & $21$ & $24$ & $\ge 27$ & $\ge 27$ & & & & & & & & & &\\
$8$ & $12$ & $14$ & $17$ & $20$ & $\ge 23$ & $\ge 26$ & $\ge 29$ & & & & & & & & & & &\\
$10$ & $15$ & $20$ & $\ge 24$ & $\ge 27$ & $\ge 28$ & & & & & & & & & & & & &\\
$12$ & $18$ & $\ge 22$ & $\ge 25$ & & & & & & & & & & & & & & &\\
$14$ & $\ge 21$ & $\ge 24$ & $\ge 26$ & & & & & & & & & & & & & & &\\
$16$ & $\ge 23$ & $\ge 26$ & & & & & & & & & & & & & & & &\\
\hline
\end{tabular}
\caption{Cyclic Ramsey numbers $R_\mathrm{cyc}(M_a^\mathrm{nest}, C_b^\mathrm{mon})$ with even $4 \le a \le 16$ and any $3 \le b \le 20$.}
\label{mnest_cmon_cyc_tab}
\end{table}

The computational results for the ordered and cyclic Ramsey numbers of nested matchings versus start-central stars are shown in Tables \ref{mnest_ssc_ord_tab} and \ref{mnest_ssc_cyc_tab}. Once again, the ordered numbers appear to behave in a complicated manner. As for the cyclic Ramsey numbers, the computed results suggest the following conjecture.

\begin{conjecture}\label{mnest_ssc_cyc_conj}
    For any even $a \ge 2$ and any $b \ge 2$, we have
    \[
        R_\mathrm{cyc}(M_a^\mathrm{nest}, S_b) = \begin{cases}
            a + b - 3, & \mbox{if $4 \mid a$ and $b$ is odd},\\
            a + b - 2, & \mbox{otherwise}.
        \end{cases}
    \]
\end{conjecture}

\begin{table}[H]
\centering
\footnotesize
\resizebox{\textwidth}{!}{
\begin{tabular}{|c||rrrrrrrrrrrrrrrrrr|}
\hline
\backslashbox{$a$}{$b$} & $3$ & $4$ & $5$ & $6$ & $7$ & $8$ & $9$ & $10$ & $11$ & $12$ & $13$ & $14$ & $15$ & $16$ & $17$ & $18$ & $19$ & $20$\\
\hline
\hline
$4$ & $6$ & $7$ & $9$ & $10$ & $11$ & $13$ & $14$ & $15$ & $16$ & $18$ & $19$ & $20$ & $21$ & $22$ & $24$ & $25$ & $26$ & $27$\\
$6$ & $8$ & $10$ & $12$ & $13$ & $15$ & $16$ & $17$ & $19$ & $20$ & $21$ & $23$ & $24$ & $25$ & $\ge 27$ & $\ge 28$ & $\ge 29$ & &\\
$8$ & $11$ & $13$ & $14$ & $16$ & $18$ & $19$ & $21$ & $22$ & $23$ & $\ge 25$ & $\ge 26$ & $\ge 28$ & & & & & &\\
$10$ & $13$ & $15$ & $17$ & $19$ & $20$ & $22$ & $\ge 24$ & $\ge 25$ & $\ge 27$ & & & & & & & & &\\
$12$ & $15$ & $18$ & $20$ & $21$ & $\ge 23$ & $\ge 25$ & $\ge 26$ & & & & & & & & & & &\\
$14$ & $18$ & $20$ & $22$ & $\ge 24$ & $\ge 26$ & $\ge 27$ & & & & & & & & & & & &\\
$16$ & $20$ & $22$ & $\ge 25$ & $\ge 27$ & $\ge 28$ & & & & & & & & & & & & &\\
\hline
\end{tabular}
}
\caption{Ordered Ramsey numbers $R_\mathrm{ord}(M_a^\mathrm{nest}, S_b^\mathrm{sc})$ with even $4 \le a \le 16$ and any $3 \le b \le 20$.}
\label{mnest_ssc_ord_tab}
\end{table}

\begin{table}[H]
\centering
\footnotesize
\resizebox{\textwidth}{!}{
\begin{tabular}{|c||rrrrrrrrrrrrrrrrrr|}
\hline
\backslashbox{$a$}{$b$} & $3$ & $4$ & $5$ & $6$ & $7$ & $8$ & $9$ & $10$ & $11$ & $12$ & $13$ & $14$ & $15$ & $16$ & $17$ & $18$ & $19$ & $20$\\
\hline
\hline
$4$ & $4$ & $6$ & $6$ & $8$ & $8$ & $10$ & $10$ & $12$ & $12$ & $14$ & $14$ & $16$ & $16$ & $18$ & $18$ & $20$ & $20$ & $22$\\
$6$ & $7$ & $8$ & $9$ & $10$ & $11$ & $12$ & $13$ & $14$ & $15$ & $16$ & $17$ & $18$ & $19$ & $20$ & $21$ & $22$ & $\ge 23$ & $\ge 24$\\
$8$ & $8$ & $10$ & $10$ & $12$ & $12$ & $14$ & $14$ & $16$ & $16$ & $\ge 18$ & $\ge 18$ & $\ge 20$ & $\ge 20$ & $\ge 22$ & $\ge 22$ & & &\\
$10$ & $11$ & $12$ & $13$ & $14$ & $15$ & $16$ & $\ge 17$ & $\ge 18$ & $\ge 19$ & $\ge 20$ & $\ge 21$ & $\ge 22$ & & & & & &\\
$12$ & $12$ & $14$ & $14$ & $16$ & $\ge 16$ & $\ge 18$ & $\ge 18$ & $\ge 20$ & $\ge 20$ & $\ge 21$ & & & & & & & &\\
$14$ & $15$ & $16$ & $17$ & $\ge 18$ & $\ge 19$ & $\ge 20$ & $\ge 21$ & $\ge 21$ & $\ge 23$ & & & & & & & & &\\
$16$ & $16$ & $18$ & $\ge 18$ & $\ge 20$ & $\ge 20$ & $\ge 22$ & $\ge 22$ & $\ge 23$ & & & & & & & & & &\\
\hline
\end{tabular}
}
\caption{Cyclic Ramsey numbers $R_\mathrm{cyc}(M_a^\mathrm{nest}, S_b)$ with even $4 \le a \le 16$ and any $3 \le b \le 20$.}
\label{mnest_ssc_cyc_tab}
\end{table}

Although we do not prove Conjecture \ref{mnest_ssc_cyc_conj}, we apply Lemma \ref{circulant_lemma} to show that the conjectured values are all lower bounds.

\begin{proposition}
    For any $a \ge 4$ divisible by four and any odd $b \ge 3$, we have $R_\mathrm{cyc}(M_a^\mathrm{nest}, S_b) \ge a + b - 3$.
\end{proposition}
\begin{proof}
    The statement trivially holds for $b = 3$, so we assume that $b \ge 5$. Let $G$ be the $2$-edge-coloring of $K_n$ with $n = a + b - 4$ such that the edge between any two distinct vertices is colored with color $2$ if and only if their circular distance is at most $\frac{b - 3}{2}$. In this case, $G$ has a circulant edge coloring, and each vertex is incident to $b - 3$ edges of color $2$ and $a - 2$ edges of color $1$. It is straightforward to see that $G$ does not contain $S_b$ as a monochromatic subgraph of color $2$. In addition, by applying Lemma \ref{circulant_lemma} to $G$ for $c = 1$, we conclude that a largest monochromatic nested matching of color $1$ contained in $G$, with the vertices in cyclic order, has $\frac{a}{2} - 1$ edges. Therefore, $G$ does not contain $M_a^\mathrm{nest}$ as a monochromatic subgraph of color~$1$, with the vertices in cyclic order. Thus, $R_\mathrm{cyc}(M_a^\mathrm{nest}, S_b) \ge a + b - 3$.
\end{proof}

\begin{proposition}
    Let $a \ge 2$ be even and $b \ge 2$ such that $a \equiv 2 \pmod 4$ or $b$ is even. Then $R_\mathrm{cyc}(M_a^\mathrm{nest}, S_b) \ge a + b - 2$.
\end{proposition}
\begin{proof}
    Let $n = a + b - 3$. First, suppose that $a \equiv 2 \pmod 4$, and let $G$ be the $2$-edge-coloring of $K_n$ such that the edge between any two distinct vertices is colored with color $1$ if and only if their circular distance is at most $\frac{a}{2} - 1$. As shown in the proof of Theorem \ref{mnest_mnest_cyc_thp}, $G$ does not contain $M_a^\mathrm{nest}$ as a monochromatic subgraph of color~$1$, with the vertices in cyclic order, when $b$ is even. The same argument applies when $b$ is odd. In addition, since each vertex of $G$ is incident to $a - 2$ edges of color~$1$ and $b - 2$ edges of color~$2$, it follows that $G$ also does not contain $S_b$ as a monochromatic subgraph of color~$2$. Therefore, $R_\mathrm{cyc}(M_a^\mathrm{nest}, S_b) \ge n + 1 = a + b - 2$.

    Now, suppose that $b$ is even, and let $G$ be the $2$-edge-coloring of $K_n$ such that the edge between any two distinct vertices is colored with color $2$ if and only if their circular distance is at most $\frac{b}{2} - 1$. In this case, each vertex of $G$ is again incident to $a - 2$ edges of color $1$ and $b - 2$ edges of color $2$, hence $G$ does not contain $S_b$ as a monochromatic subgraph of color $2$. By applying Lemma \ref{circulant_lemma} to $G$ for $c = 1$, it follows that a largest monochromatic nested matching of color~$1$ contained in $G$, with the vertices in cyclic order, has $\frac{a}{2} - 1$ edges. Therefore, $G$ does not contain $M_a^\mathrm{nest}$ as a monochromatic subgraph of color~$1$, with the vertices in cyclic order, implying $R_\mathrm{cyc}(M_a^\mathrm{nest}, S_b) \ge n + 1 = a + b - 2$.
\end{proof}

We conclude the section by investigating the ordered and cyclic Ramsey numbers of nested matchings versus complete graphs. The resulting values are given in Tables \ref{mnest_k_ord_tab} and \ref{mnest_k_cyc_tab}. Although the ordered numbers do not seem to exhibit an easily discernible structure, the following conjecture on the cyclic Ramsey numbers naturally arises.

\begin{conjecture}\label{mnest_k_cyc_conj}
    For any even $a \ge 2$ and any $b \ge 2$, we have $R_\mathrm{cyc}(M_a^\mathrm{nest}, K_b) = \frac{ab}{2}$.
\end{conjecture}

\begin{table}[H]
\centering
\footnotesize
\begin{tabular}{|c||rrrrrrrr|}
\hline
\backslashbox{$a$}{$b$} & $3$ & $4$ & $5$ & $6$ & $7$ & $8$ & $9$ & $10$\\
\hline
\hline
$4$ & $7$ & $10$ & $13$ & $16$ & $19$ & $22$ & $25$ & $\ge 28$\\
$6$ & $11$ & $16$ & $21$ & $26$ & $\ge 31$ & $\ge 36$ & &\\
$8$ & $16$ & $23$ & $\ge 31$ & $\ge 36$ & $\ge 42$ & & &\\
$10$ & $20$ & $\ge 29$ & $\ge 35$ & & & & &\\
$12$ & $\ge 25$ & $\ge 31$ & & & & & &\\
$14$ & $\ge 27$ & & & & & & &\\
$16$ & $\ge 30$ & & & & & & &\\
\hline
\end{tabular}
\caption{Ordered Ramsey numbers $R_\mathrm{ord}(M_a^\mathrm{nest}, K_b)$ with even $4 \le a \le 16$ and any $3 \le b \le 10$.}
\label{mnest_k_ord_tab}
\end{table}

\begin{table}[H]
\centering
\footnotesize
\begin{tabular}{|c||rrrrrrrr|}
\hline
\backslashbox{$a$}{$b$} & $3$ & $4$ & $5$ & $6$ & $7$ & $8$ & $9$ & $10$\\
\hline
\hline
$4$ & $6$ & $8$ & $10$ & $12$ & $14$ & $16$ & $18$ & $20$\\
$6$ & $9$ & $12$ & $15$ & $18$ & $21$ & $\ge 24$ & $\ge 26$ &\\
$8$ & $12$ & $16$ & $20$ & $\ge 23$ & $\ge 26$ & $\ge 28$ & &\\
$10$ & $15$ & $\ge 20$ & $\ge 24$ & $\ge 26$ & & & &\\
$12$ & $18$ & $\ge 23$ & $\ge 25$ & $\ge 27$ & & & &\\
$14$ & $\ge 21$ & $\ge 25$ & $\ge 26$ & & & & &\\
$16$ & $\ge 23$ & $\ge 27$ & & & & & &\\
\hline
\end{tabular}
\caption{Cyclic Ramsey numbers $R_\mathrm{cyc}(M_a^\mathrm{nest}, K_b)$ with even $4 \le a \le 16$ and any $3 \le b \le 10$.}
\label{mnest_k_cyc_tab}
\end{table}

\section{Reinforcement learning}\label{sc_rl}

In addition to using Kissat to compute the ordered and cyclic Ramsey numbers presented in Section \ref{sc_results}, we also demonstrate how RL can be used to obtain lower bounds on Ramsey numbers. The idea of applying RL to extremal graph theory was recently introduced in the pioneering work of Wagner \cite{Wagner2021} and has since attracted considerable attention.

In particular, Ghebleh et al.\ \cite{GheYaKaSte2026} reimplemented Wagner's original approach to improve its stability and readability, and subsequently used the new framework to obtain a series of theoretical results \cite{GheYaKaSte2025A, GheYaKaSte2025B, GheYaKaSte2026}. These include new lower bounds on small standard Ramsey numbers involving complete bipartite graphs, wheel graphs and book graphs \cite{Radziszowski}. Concurrently, Angileri et al.\ \cite{Angileri2025A} systematized the previous work by implementing four distinct RL environments specialized for graph theory, and later used their framework \cite{Angileri2025B} to contribute to the study of Brouwer's conjecture \cite{BrouHae2012}.

In this work, we employ the recently developed RLGT framework \cite{RLGT}, which builds upon previous approaches through its ease of use, clean interface and modular design. We use the Deep Cross-Entropy agent together with the Linear Build environment, as this combination achieved the strongest performance in the reported experimental results.

In our setting, the agent iteratively interacts with the environment to construct a $2$-edge-coloring of $K_n$ by executing $\binom{n}{2}$ binary actions. Each action determines whether the next edge is colored with color~$1$ or color~$2$, with the edges arranged in a predetermined order, such as the row-major order described in Section \ref{sc_sat}. The action selection mechanism is governed by a configurable policy network. An episode consisting of $\binom{n}{2}$ actions produces a single $2$-edge-coloring of $K_n$.

Each constructed graph is assigned a score value, and the agent aims to learn how to build colorings that maximize this score. The Deep Cross-Entropy method \cite{BoKroMaRu2005, Rubinstein1997} follows an evolutionary strategy: in each iteration, it generates a population of graphs and selects a configurable number of top-performing episodes. These episodes are then used to train the policy network, which influences how future actions are selected. Additionally, another subset of top-performing episodes can be carried over to the next generation.

For more details on standard RL terminology, the reader can refer to the standard textbooks \cite{Lapan2020, Powell2022, SuBa2018, Szepesvari2010}. Further technical details on using the RLGT framework are available in the framework documentation \cite{RLGTDoc}.

With the framework in place, the remaining step is to define the score function. As noted by Ghebleh et al.\ \cite{GheYaKaSte2025B}, in the context of standard Ramsey numbers, a natural choice is
\begin{equation}\label{sfun_std}
    -\sum_{(w_0, w_1, w_2, \ldots, w_{|H_1| - 1}) \in \mathcal{S}_1} \Iverson{ \bigwedge_{uv \in E(H_1)} \neg x_{w_u, w_v} } - \sum_{(w_0, w_1, w_2, \ldots, w_{|H_2| - 1}) \in \mathcal{S}_2} \Iverson{ \bigwedge_{uv \in E(H_2)} x_{w_u, w_v} },
\end{equation}
where the $x_{u, v}$ are the Boolean variables described in Section \ref{sc_sat}, $\mathcal{S}_1$ (resp.\ $\mathcal{S}_2$) comprises all the $|H_1|$-tuples (resp.\ $|H_2|$-tuples) of elements in $V(K_n)$ with mutually distinct entries, and $\Iverson\cdot$ denotes the Iverson bracket, i.e.,
\[
    \Iverson P = \begin{cases}
        1, & \mbox{if $P$ holds}, \\
        0, & \mbox{otherwise}.
    \end{cases}
\]

Essentially, \eqref{sfun_std} counts the number of forbidden embeddings of $H_1$ and $H_2$ in a given $2$-edge-coloring of $K_n$, with the negative signs converting a minimization problem into a maximization one. With this choice of score function, the RLGT framework can be used to heuristically search for edge colorings that maximize the score. If the best score equals zero, then we have found a $2$-edge-coloring of $K_n$ that avoids both forbidden monochromatic subgraphs, which yields a lower bound on the corresponding Ramsey number. If the attained best score is negative, then no useful information is obtained. Due to the nondeterministic nature of parameter initialization and action selection, multiple runs may sometimes be required to achieve the desired result.

When investigating ordered and cyclic Ramsey numbers, the same idea can be applied with minor modifications to the score function. For ordered Ramsey numbers, we use
\[
    -\sum_{(w_0, w_1, w_2, \ldots, w_{|H_1| - 1}) \in \mathcal{O}_1} \Iverson{ \bigwedge_{uv \in E(H_1)} \neg x_{w_u, w_v} } - \sum_{(w_0, w_1, w_2, \ldots, w_{|H_2| - 1}) \in \mathcal{O}_2} \Iverson{ \bigwedge_{uv \in E(H_2)} x_{w_u, w_v} },
\]
while for cyclic Ramsey numbers, we use
\[
    -\sum_{(w_0, w_1, w_2, \ldots, w_{|H_1| - 1}) \in \mathcal{C}_1} \Iverson{ \bigwedge_{uv \in E(H_1)} \neg x_{w_u, w_v} } - \sum_{(w_0, w_1, w_2, \ldots, w_{|H_2| - 1}) \in \mathcal{C}_2} \Iverson{ \bigwedge_{uv \in E(H_2)} x_{w_u, w_v} } ,
\]
where $\mathcal{O}_1$, $\mathcal{O}_2$, $\mathcal{C}_1$ and $\mathcal{C}_2$ are defined as in Section \ref{sc_sat}. The only difference is that we consider only the appropriate forbidden embeddings:\ increasing embeddings for ordered Ramsey numbers, and the embeddings increasing up to a cyclic permutation for cyclic Ramsey numbers.

We provide the \texttt{Python} RL training script \texttt{rlgt\_trainer.py}, located in the folder \texttt{src} in \cite{GitHub}. The script uses RLGT to perform agent--environment interactions and computes the Ramsey score functions through an external dynamically linked library implemented in \texttt{C++}. In this library, the main function accepts an array of graphs $G_1, G_2, \ldots, G_p$ and, for each graph $G_j$, counts the number of embeddings $\varphi \colon V(H) \to V(G_j)$, where $H$ is a given pattern graph.

Additionally, the function can be configured to count only increasing embeddings, or embeddings that are increasing up to a cyclic permutation, thereby providing direct support for ordered and cyclic Ramsey number problems. For each $j \in \{ 1, 2, \ldots, p \}$, the function recursively generates all sequences corresponding to possible embeddings $\varphi \colon V(H) \to V(G_j)$ according to the selected counting mode.

The implementation supports graphs of order at most $64$. Graphs are represented using a $64 \times 64$ adjacency matrix stored in bitmask format as $64$ unsigned $64$-bit integers. Bitwise operations on these integers are used to check whether a partially constructed embedding is valid. If the procedure detects that a partial embedding cannot be extended to a full embedding, the recursive call terminates early. The function then returns the total number of full embeddings.

To improve performance, the computation is parallelized. For each of the graphs $G_1, G_2, \ldots, G_p$, a separate task is created and executed within a thread pool to maximize CPU utilization on modern multi-core systems. For more technical details, the reader can refer to the folder \texttt{src/score\_computation} in \cite{GitHub}.

\begin{example}
    By executing the \texttt{Python} RL training script, it is not difficult to find a $2$-edge-coloring of $K_{11}$ that does not contain $P_6^\mathrm{alt}$ as a monochromatic subgraph in the ordered Ramsey number problem. This yields the best possible lower bound $R_\mathrm{ord}(P_6^\mathrm{alt}, P_6^\mathrm{alt}) \ge 12$. By performing analogous computations, we can also obtain the lower bounds $R_\mathrm{ord}(P_6^\mathrm{alt}, P_7^\mathrm{alt}) \ge 13$, $R_\mathrm{ord}(P_6^\mathrm{alt}, P_8^\mathrm{alt}) \ge 14$ and $R_\mathrm{ord}(P_7^\mathrm{alt}, P_7^\mathrm{alt}) \ge 14$. However, these three lower bounds are not optimal, since $R_\mathrm{ord}(P_6^\mathrm{alt}, P_7^\mathrm{alt}) = 14$, $R_\mathrm{ord}(P_6^\mathrm{alt}, P_8^\mathrm{alt}) = 15$ and $R_\mathrm{ord}(P_7^\mathrm{alt}, P_7^\mathrm{alt}) = 15$; see Table~\ref{palt_palt_ord_tab}. If we attempt to prove the corresponding optimal lower bounds, the training procedure typically gets stuck at a best attained reward of $-1$ or $-2$ after sufficiently many learning iterations. Therefore, in these cases the RL-based approach fails and is outperformed by the SAT-based approach. \hfill$\Diamond$
\end{example}

\section{Conclusion}\label{sc_conclusion}

We have introduced the cyclic Ramsey numbers as a natural relaxation of the ordered Ramsey numbers, and computed new small two-color ordered and cyclic Ramsey numbers for several classes of graphs: monotone paths, monotone cycles, alternating paths, reverse alternating paths, start-central stars, complete graphs and nested matchings. All of these numbers were obtained by solving the corresponding SAT problems using the Kissat SAT solver. The computational results then helped us determine all ordered or cyclic Ramsey numbers for several pairs of classes of graphs, together with some bounds on the ordered and cyclic Ramsey numbers where one argument is a connected graph, while the other is a monotone path or a monotone cycle.

Here, we note that while Kissat was typically able to find good lower bounds for both ordered and cyclic Ramsey numbers, it was much more successful in finding upper bounds for the ordered case. This can be explained by the fact that SAT instances for cyclic Ramsey numbers tend to have a noticeably greater number of clauses.

We also explored how RL can be applied to ordered and cyclic Ramsey number problems using the RLGT framework. While RL can sometimes provide useful lower bounds, it does not always find the optimal bound, even for Ramsey numbers below $20$. Moreover, the RL-based approach is not capable of finding upper bounds at all. In contrast, the SAT-based approach typically finds optimal lower bounds below $20$ with ease, and in many cases it can compute the exact Ramsey number, especially for the ordered case. For larger Ramsey numbers, Kissat may struggle to find a precise upper bound, although it often still provides a valid lower bound. Despite these limitations, the partial success of the RL-based approach suggests that it could still be a valuable tool for tackling similar combinatorial problems in the future.

We believe that cyclic Ramsey numbers could be a fruitful topic for future research. While much is known about the asymptotic behavior of standard and ordered Ramsey numbers, the asymptotic behavior of cyclic Ramsey numbers remains largely unexplored. By Proposition \ref{cyc_basic_prop}, a natural question arises regarding whether cyclic Ramsey numbers behave asymptotically more like ordered Ramsey numbers or standard Ramsey numbers. For instance, Theorems \ref{pmon_con_th} and \ref{pmon_con_th_2} show that when one argument is a monotone path and the other is a connected graph, cyclic Ramsey numbers behave asymptotically like ordered Ramsey numbers up to a constant factor. A natural research direction is to investigate whether the same holds for other classes of graphs. We believe that a good starting point for this would be to verify or refute some of Conjectures \ref{pmon_pmon_cyc_conj}--\ref{palt_pralt_cyc_conj}, \ref{cmon_palt_ord_conj}, \ref{cmon_palt_cyc_conj}, \ref{s_pmon_cyc_conj}, \ref{s_palt_cyc_conj}, \ref{k_palt_ord_conj}, \ref{k_palt_cyc_conj}, \ref{mnest_mnest_cyc_conj}--\ref{mnest_ssc_cyc_conj} and \ref{mnest_k_cyc_conj}.

We now show how the different Ramsey-type formulations involving standard, ordered and cyclic Ramsey numbers can be unified within a group-theoretic framework. First, we extend the notion of rotational isomorphism in the following manner. Let $\Sym(n)$ denote the symmetric group on $n$ elements, i.e., the group of all permutations of the set $\{ 0, 1, 2, \ldots, n - 1 \}$, where the group operation is the composition of permutations. For any two graphs $G_1$ and $G_2$ of the same order, and any group $\Gamma \le \Sym(|G_1|)$, we write $G_1 \cong_\Gamma G_2$ if there exists an isomorphism $\varphi \colon V(G_1) \to V(G_2)$ such that $\varphi \in \Gamma$. The relation $\cong_\Gamma$ defined in this way is clearly an equivalence relation. Moreover, given two graphs $G_1$ and $G_2$ with $|G_1| = |G_2| = n$, the relation $G_1 \cong_\rho G_2$ is now equivalent to $G_1 \cong_\Gamma G_2$ with $\Gamma = \left\langle \begin{pmatrix}\begin{smallmatrix}
    0 & 1 & 2 & \cdots & n - 2 & n - 1\\
    1 & 2 & 3 & \cdots & n - 1 & 0
\end{smallmatrix}\end{pmatrix} \right\rangle$, where $\begin{pmatrix}\begin{smallmatrix}
    0 & 1 & 2 & \cdots & n - 2 & n - 1\\
    1 & 2 & 3 & \cdots & n - 1 & 0
\end{smallmatrix}\end{pmatrix}$ is the cyclic shift by one position to the left.

Now, consider graphs $H_1, H_2, \ldots, H_k$ with $k \in \mathbb{N}$ and groups $\Gamma_1, \Gamma_2, \ldots, \Gamma_k$ such that $\Gamma_j \le \Sym(|H_j|)$ for each $j \in \{ 1, 2, \ldots, k \}$. We define the \emph{permutational Ramsey number} $R(H_1^{\Gamma_1}, H_2^{\Gamma_2}, \ldots, H_k^{\Gamma_k})$ as the smallest $n \in \mathbb{N}$ such that every $k$-edge-coloring of $K_n$ contains a monochromatic subgraph $H_j' \cong_{\Gamma_j} H_j$ of color $j$ for some $j \in \{ 1, 2, \ldots, k \}$, with the vertices appearing in the same order as in $H_j'$. By Ramsey's theorem, all of these numbers are well-defined. With this in mind, the standard, ordered and cyclic Ramsey numbers can be viewed as permutational Ramsey numbers corresponding to specific choices of $\Gamma_1, \Gamma_2, \ldots, \Gamma_k$.
\begin{enumerate}[label=\textbf{(\arabic*)}]
    \item The ordered Ramsey number $R_\mathrm{ord}(H_1, H_2, \ldots, H_k)$ is the permutational Ramsey number $R(H_1^{\Gamma_1}, H_2^{\Gamma_2}, \ldots, \linebreak H_k^{\Gamma_k})$, where all the groups $\Gamma_1, \Gamma_2, \ldots, \Gamma_k$ are trivial and contain only the identity permutation.
    \item The standard Ramsey number $R(H_1, H_2, \ldots, H_k)$ is the permutational Ramsey number $R(H_1^{\Gamma_1}, H_2^{\Gamma_2}, \ldots, \linebreak H_k^{\Gamma_k})$, where all the groups $\Gamma_1, \Gamma_2, \ldots, \Gamma_k$ contain all possible permutations, i.e., $\Gamma_j = \Sym(|H_j|)$ for each $j \in \{ 1, 2, \ldots, k \}$.
    \item The cyclic Ramsey number $R_\mathrm{cyc}(H_1, H_2, \ldots, H_k)$ is the permutational Ramsey number $R(H_1^{\Gamma_1}, H_2^{\Gamma_2}, \ldots, \linebreak H_k^{\Gamma_k})$, where all the groups $\Gamma_1, \Gamma_2, \ldots, \Gamma_k$ are cyclic and of the form $\Gamma_j = \left\langle \begin{pmatrix}\begin{smallmatrix}
    0 & 1 & 2 & \cdots & |H_j| - 2 & |H_j| - 1\\
    1 & 2 & 3 & \cdots & |H_j| - 1 & 0
\end{smallmatrix}\end{pmatrix} \right\rangle$ for each $j \in \{ 1, 2, \ldots, k \}$.
\end{enumerate}

We believe that it is natural to ask how different choices of $\Gamma_1, \Gamma_2, \ldots, \Gamma_k$ affect the permutational Ramsey numbers. Moreover, by selecting $\Gamma_1, \Gamma_2, \ldots, \Gamma_k$ according to a specific pattern, it is possible to construct several additional natural generalizations of Ramsey numbers.

\begin{definition}
    The \emph{reflective Ramsey number} $R_\mathrm{ref}(H_1, H_2, \ldots, H_k)$ is the permutational Ramsey number $R(H_1^{\Gamma_1}, H_2^{\Gamma_2}, \ldots, H_k^{\Gamma_k})$, where each group $\Gamma_j$ consists of the identity permutation and the reflection, i.e., $\Gamma_j = \left\langle \begin{pmatrix}\begin{smallmatrix}
    0 & 1 & 2 & \cdots & |H_j| - 1\\
    |H_j|-1 & |H_j|-2 & |H_j|-3 & \cdots & 0
\end{smallmatrix}\end{pmatrix} \right\rangle$ for each $j \in \{ 1, 2, \ldots, k \}$.
\end{definition}
\begin{definition}
    The \emph{dihedral Ramsey number} $R_\mathrm{dih}(H_1, H_2, \ldots, H_k)$ is the permutational Ramsey number $R(H_1^{\Gamma_1}, H_2^{\Gamma_2}, \ldots, H_k^{\Gamma_k})$, where all the groups $\Gamma_1, \Gamma_2, \ldots, \Gamma_k$ are dihedral and of the form
    \[
        \Gamma_j = \left\langle \begin{pmatrix}\begin{smallmatrix}
    0 & 1 & 2 & \cdots & |H_j| - 2 & |H_j| - 1\\
    1 & 2 & 3 & \cdots & |H_j| - 1 & 0
\end{smallmatrix}\end{pmatrix}, \begin{pmatrix}\begin{smallmatrix}
    0 & 1 & 2 & \cdots & |H_j| - 1\\
    |H_j|-1 & |H_j|-2 & |H_j|-3 & \cdots & 0
\end{smallmatrix}\end{pmatrix} \right\rangle
    \]
    for each $j \in \{ 1, 2, \ldots, k \}$. 
\end{definition}
\begin{definition}
    The \emph{alternating Ramsey number} $R_\mathrm{alt}(H_1, H_2, \ldots, H_k)$ is the permutational Ramsey number $R(H_1^{\Gamma_1}, H_2^{\Gamma_2}, \ldots, H_k^{\Gamma_k})$, where $\Gamma_j$ is the alternating group on the set $\{ 0, 1, 2, \ldots, |H_j| - 1 \}$ for each $j \in \{ 1, 2, \ldots, k \}$.
\end{definition}

\noindent
The next observation immediately follows.

\begin{proposition}\label{ref_dih_alt_prop}
For any graphs $H_1, H_2, \ldots, H_k$ with $k \in \mathbb{N}$, we have
\begin{equation}\label{good_chain}
    R(H_1, \ldots, H_k) \le R_\mathrm{dih}(H_1, \ldots, H_k) \le R_\mathrm{cyc}(H_1, \ldots, H_k), R_\mathrm{ref}(H_1, \ldots, H_k) \le R_\mathrm{ord}(H_1, \ldots, H_k),
\end{equation}
and $R(H_1, \ldots, H_k) \le R_\mathrm{alt}(H_1, \ldots, H_k) \le R_\mathrm{ord}(H_1, \ldots, H_k)$.
\end{proposition}

\noindent
We conclude the paper with the following example, which demonstrates that each inequality from \eqref{good_chain} can be strict.

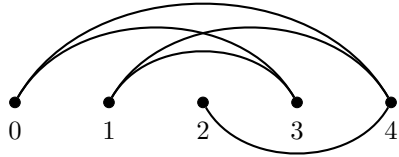
\begin{figure}[t]
\centering
\begin{tikzpicture}
    \tikzstyle{vertex}=[draw,circle,font=\scriptsize,minimum size=4pt,inner sep=1pt,fill=black]
    \tikzstyle{edge}=[draw,thick]

    \foreach \i in {0,1,2,3,4} {
        \node[vertex] (v\i) at ({1.25*\i}, 0) {};
        \node[below=4pt] at (v\i) {$\i$};
    }

    \path[edge, bend left=60] (v0) to (v3);
    \path[edge, bend left=60] (v0) to (v4);
    \path[edge, bend left=60] (v1) to (v3);
    \path[edge, bend left=60] (v1) to (v4);
    \path[edge, bend left=60] (v4) to (v2);
\end{tikzpicture}
\caption{The graph $H$ from Example \ref{ugly_example}.}
\label{ugly_fig}
\end{figure}

\begin{example}\label{ugly_example}
The \texttt{src/sat\_solving} folder in \cite{GitHub} contains a \texttt{Python} script for generating SAT instances, which supports reflective and dihedral Ramsey numbers in addition to the ordered and cyclic ones. By executing this script several times together with Kissat, it is not difficult to compute that $R_\mathrm{ord}(P_7^\mathrm{alt}, P_7^\mathrm{alt}) = 15$, $R_\mathrm{ref}(P_7^\mathrm{alt}, P_7^\mathrm{alt}) = 14$ and $R_\mathrm{cyc}(P_7^\mathrm{alt}, P_7^\mathrm{alt}) = R_\mathrm{dih}(P_7^\mathrm{alt}, P_7^\mathrm{alt}) = 11$.

Now, let $H$ be the unicyclic graph of order five shown in Figure \ref{ugly_fig}. Using the same two programs, we obtain $R_\mathrm{ord}(H, H) = R_\mathrm{ref}(H, H) = 12$, $R_\mathrm{cyc}(H, H) = 10$ and $R_\mathrm{dih}(H, H) = 9$. This means that each inequality from~\eqref{good_chain} can be strict, even in the diagonal case. Moreover, since the groups associated with the cyclic Ramsey numbers tend to have a greater order than those for the reflective Ramsey numbers, it is natural to expect that the cyclic Ramsey numbers are smaller than the reflective ones. \hfill$\Diamond$
\end{example}

\section*{Acknowledgments}

N.\ Bašić is supported in part by the Slovenian Research Agency (research program P1-0294). I.\ Damnjanović is supported by the Ministry of Science, Technological Development and Innovation of the Republic of Serbia, grant number 451-03-34/2026-03/200102, and the Science Fund of the Republic of Serbia, grant \#6767, Lazy walk counts and spectral radius of threshold graphs --- LZWK.

\section*{Conflict of interest}

The authors declare that they have no conflict of interest.

\end{document}